\documentclass[10pt]{article} 
\usepackage{geometry}
\usepackage{parskip}
\usepackage{amssymb, amsmath, amsthm}
\usepackage{amsfonts}
\numberwithin{equation}{section}
\usepackage{subcaption}
\usepackage[colorlinks=true,linkcolor=blue,anchorcolor=blue,citecolor=black,urlcolor=blue,plainpages=false,pdfpagelabels]{hyperref}
\UseRawInputEncoding 
\usepackage[longnamesfirst,authoryear,round]{natbib}
\usepackage[dvipsnames]{xcolor}
\usepackage{xy}\xyoption{all} \xyoption{poly} \xyoption{knot}
\usepackage{float}
\usepackage{bm}
\usepackage{bbm}
\usepackage{multicol,multirow}
\usepackage{array}
\usepackage{relsize}
\usepackage{chngcntr}
\usepackage{etoolbox}
\usepackage{caption}
\usepackage{tikz}
\usetikzlibrary{decorations.pathreplacing,calc}
\usetikzlibrary{shapes,backgrounds}
\usetikzlibrary{patterns}
\usetikzlibrary{cd}
\usepackage{amsopn}
\usepackage{calc}
\usepackage{algorithm}
\usepackage[noend]{algpseudocode}
\usepackage{mathtools}

\usepackage{tikz-cd} 
\usepackage{amscd} 
\usepackage{comment}
\usepackage{physics}

\newcounter{cdrow}

\newtheorem{theorem}{Theorem}[]
\newtheorem*{theorem*}{Theorem}

\newtheorem{proposition}[theorem]{Proposition}

\newtheorem*{claim*}{Claim}

\theoremstyle{definition}
\newtheorem{definition}[theorem]{Definition}
\newtheorem*{definition*}{Definition}

\theoremstyle{remark}
\newtheorem{remark}[theorem]{Remark}

\newtheorem{example}[theorem]{Example}
\newtheorem*{example*}{Example}

\newcommand{\R}{\mathbb{R}}

\newcommand{\Pers}{\mathrm{Pers}}
\newcommand{\birth}{\mathrm{birth} \,}
\newcommand{\death}{\mathrm{death} \,}

\renewcommand{\Im}{\mathrm{Im} \,}

\renewcommand{\H}{\mathrm{H}}
\renewcommand{\rank}{\mathrm{rank} \,}

\newcommand{\VR}{\mathrm{VR}}

\title{\Large\bf{Fast Topological Signal Identification and Persistent Cohomological Cycle Matching}}

\date{} 					

\author{In\'{e}s Garc\'{i}a-Redondo\thanks{Department of Mathematics, Imperial College London, UK} \thanks{London School of Geometry and Number Theory, University College London, UK} \\
\texttt{i.garcia-redondo22@imperial.ac.uk}
\and
Anthea Monod\footnotemark[1]\\
\texttt{a.monod@imperial.ac.uk}
\and
Anna Song\footnotemark[1] \thanks{Haematopoietic Stem Cell Laboratory, The Francis Crick Institute, London, UK}\\
\texttt{a.song19@imperial.ac.uk}}
\date{}

\begin{document}
\def\spacingset#1{\renewcommand{\baselinestretch}%
{#1}\small\normalsize} \spacingset{1}
\maketitle


\section*{Abstract}
Within the context of topological data analysis, the problems of identifying topological significance and matching signals across datasets are important and useful inferential tasks in many applications. The limitation of existing solutions to these problems, however, is computational speed. In this paper, we harness the state-of-the-art for persistent homology computation by studying the problem of determining topological prevalence and cycle matching using a cohomological approach, which increases their feasibility and applicability to a wider variety of applications and contexts. We demonstrate this on a wide range of real-life, large-scale, and complex datasets. We extend existing notions of topological prevalence and cycle matching  to include general non-Morse filtrations. This provides the most general and flexible state-of-the-art adaptation of topological signal identification and persistent cycle matching, which performs comparisons of orders of ten for thousands of sampled points in a matter of minutes on standard institutional HPC CPU facilities.

\paragraph{Keywords:} Absolute (co)homology; cycle matching; cycle prevalence; image-persistent homology; relative (co)homology.

\section*{Introduction}

\textit{Persistent homology}, one of the cornerstones of topological data analysis, studies the lifespan of the topological features in a nested sequence of topological spaces by tracking the changes in its homology groups. It provides a robust statistical summary of data, capturing its ``shape'' and ``size'' and has been been applied to many scientific disciplines in recent years with great success. The diagram consisting of the homology groups of the filtration connected by the maps induced by the inclusions is called the \textit{persistence module}. From this, the \textit{persistence barcode} (or simply, \emph{barcode}) is derived---a canonical summary of the aforementioned lifespans as a set of half-open intervals.

A natural question is whether it is possible to compare the barcodes obtained from different filtrations, which would, for instance, provide a correspondence between some of their intervals. Several solutions have been proposed. \cite{gonzalez-diaz_basis-independent_2020} derive a basis-independent partial matching for ladder modules and zigzag persistence. A different method of persistent extension to find analogous bars---especially interesting if there is no known mapping between the persistence modules---was very recently introduced by \cite{yoon_persistent_2022}. \cite{bauer_induced_2015} match intervals of barcodes using a known mapping between the persistence modules. This notion was recently reinterpreted in statistical terms by \cite{reani_cycle_2021},  who propose a similar interval matching using \textit{image-persistence}, which was first introduced by \cite{cohen-steiner_persistent_2009}. This matching is applied to define a \textit{prevalence score}---a measure of the significance of a given interval in a barcode. Typically, \emph{persistence} (i.e., the length of the interval) is interpreted as topological significance or signal: longer intervals correspond to ``true'' features, while short ones are attributed to topological noise. However, this practice can be misleading since persistence is highly affected by the distance of sampling points and usually has a higher value for cycles created at a larger scale. The prevalence score proposed by \cite{reani_cycle_2021} bypasses this shortcoming by taking into account the statistical heuristics of the problem: it is obtained by matching persistence intervals across diagrams of several resamplings of the data. 

A limitation common to all previously proposed barcode comparison techniques is that they are all computationally very expensive, which significantly limits their practicality in many applications and to many real datasets.  In this paper, we address this specific issue by leveraging the current state-of-the-art in persistent homology computation, \emph{Ripser} \citep{bauer_ripser_2021}, which studies the dual perspective and computes persistent \emph{cohomology}, thus taking advantage of its equivalence to persistent homology \citep{silva_dualities_2011}.  Furthermore, recently, Ripser was adapted to the setting of image-persistence via Ripser-image \citep{bauer_efficient_2022}.  We apply this technology to the interval matching approach proposed by \cite{reani_cycle_2021}, as well as specialize and extend their definitions to allow for greater flexibility and applicability.  The final result of our contributions is state-of-the-art software for interval matching, executable in a matter of minutes using only standard institutional high performance computing facilities, which we showcase on a wide variety of complex and large-scale datasets, such as static and time-lapse imaging and video data from biomedical and astrophysical applications.

\subsubsection*{Contributions}
\begin{itemize}
    \item We specialize the definition of interval matching proposed by \cite{reani_cycle_2021} to \emph{simplex-wise filtrations} (see Definition \ref{def:simplexwise_filtration}), making it compatible with the output of Ripser-image \citep{bauer_efficient_2022} and more widely applicable.
    \item We present a comprehensive case study 
    of different definitions for a matching affinity score
    that extends the original score
    proposed by \cite{reani_cycle_2021}.
    \item We provide state-of-the-art code for interval matching, freely available at \url{https://github.com/inesgare/interval-matching}. 
    \item We comprehensively showcase and demonstrate representative applications of our specialized definitions and code to complex and large-scale datasets.
\end{itemize} 

\subsubsection*{Outline}

We begin by introducing the fundamentals of persistent homology and set relevant notations in section \ref{sec:preliminaries}. 
We also review image-persistence and use it to present interval matching as proposed in \cite{reani_cycle_2021}. 
In section \ref{sec:cycle_matching_cohomology}, we adapt the definition of image-persistence to the various homology settings and study how these frameworks are related. Here, we propose our specialized definition of interval matching and revisit the notion of matching affinity by \cite{reani_cycle_2021} in a case study of alternative formulations.
In section \ref{sec:applications}, we present applications of the notion of cycle matching to a variety of data sets diverse in nature and aimed at different objectives. We close with a discussion of our contributions and proposals for future work in section \ref{sec:end}.

\section{Preliminaries}
\label{sec:preliminaries}

In this section we introduce the fundamental concepts underlying our work and establish some relevant notations that we will use throughout the rest of the paper.

\subsection{The Four Standard Persistence Modules}
\label{subsec:standard_persistence_modules}

A \textit{filtration} is a family of nested subspaces \(\{X_t: t \in T\}\) of some space $X$
$$X_t \subset X_s \subset X, \quad \mathrm{for}\ t \leq s, $$
where \(T \subset \R\) is a totally ordered indexing set. In this paper, we work with \textit{filtered complexes}, specifically, we further assume that $X$ is a finite simplicial complex and the spaces $X_t$ are simplicial subcomplexes of $X$. Filtered complexes can also be interpreted as diagrams \(X_\bullet: T \to \mathbf{\mathrm{Simp}}\) of simplicial complexes indexed over some finite totally ordered set \(T\), such that all maps in the diagram are inclusions.

A \textit{re-indexing} of a filtration changes the indexing set \(T\) to another totally ordered set \(I\) using some monotonic map \(r: I \to T\) so that \(X_{r(i)} = X_t\). For instance, if \(\{X_t: t \in T\}\) is a filtered complex, \(T = \{t_1,\ldots,t_n\}\) is finite  and we have the re-indexing \(r(i) = t_i\) that allows for a reparameterization of the filtration over the natural numbers \(\{X_i: 1\leq i \leq n\}\).

Applying the corresponding homology functor to the simplicial complexes in a filtered complex and the inclusions \(X_i \subset X_{i+1}\) between consecutive spaces gives us the following diagrams
\begin{alignat}{10}
    \H_*(X_\bullet):\ & & & \ \H_*(X_1) & \ \rightarrow \  & \ldots & \ \rightarrow \  & \ \H_*(X_{n-1}) & \ \rightarrow \ & \ \H_*(X_{n}), \label{eq:ah} \\[2pt]
	\H^*(X_\bullet):\ & & & \ \H^*(X_1) & \ \leftarrow \ & \ldots & \ \leftarrow \ & \ \H^*(X_{n-1}) & \ \leftarrow \ & \ \H^*(X_{n})  , \label{eq:ac}\\[2pt]
	\H_*(X, X_\bullet):\ & \ \H_*(X_n) & \ \rightarrow \ & \ \H_*(X, X_1) & \ \rightarrow \ & \ldots & \ \rightarrow \ &  \ \H_*(X, X_{n-1}), & &  \label{eq:rh}\\[2pt]
	\H^*(X, X_\bullet):\ & \ \H^*(X_n) & \ \leftarrow \ & \ \H^*(X, X_1)& \ \leftarrow \ & \ldots & \leftarrow & \ \H^*(X, X_{n-1}). \label{eq:rc}
\end{alignat}
Following \cite{silva_dualities_2011}, we will call these the \emph{four standard persistence modules}. The first persistence module \eqref{eq:ah} corresponds to \textit{absolute homology} and is the one most often used. The expressions following are the persistence modules for \textit{absolute cohomology \eqref{eq:ac}, relative homology \eqref{eq:rh},} and \emph{relative cohomology \eqref{eq:rc}}. 
Unless otherwise stated, homology and cohomology will have field coefficients, so that these persistence modules are made up of vector spaces and linear maps.

The assumption of field coefficients allows us to invoke the \emph{structure theorem} \citep{zomorodian_computing_2005}. This is one of the foundational results in persistent homology and ensures that, up to isomorphism, any persistence module, such as the ones above, can be decomposed in a direct sum of \textit{interval modules}. An interval module consists of copies of the field of coefficients over an interval range of indices; these copies are connected by the identity map and the trivial vector space outside of that interval. This allows for the interpretation that some (co)cycle is \emph{born} at the beginning of the interval and \emph{dies} at the end of it. For instance, for the absolute homology module, 
\[\H_*(X_\bullet) \cong \bigoplus_{m = 1}^M I_{[b_m,\, d_m]},\]
where the sub-index denotes the range of indices over which the interval module is nontrivial.

The collection of intervals that appears in the decomposition of the structure theorem is an invariant of the isomorphism type of the persistence module. This collection is the \emph{persistence barcode} of the filtration
\[
\Pers(\H_*(X_\bullet)) = \big\{ [b_m,\, d_m] \big\}_{m = 1}^M.
\]
The intervals from the barcode are called \textit{persistence intervals} and the start and end points of the intervals are the \textit{persistence pairs}.

The persistence barcode provides a summary of the lifespans of the topological features of the filtration.
Persistence pairs are often interpreted with real indices \(t_{b_m}\) and \(t_{d_m}\) associated to the natural indices \(b_m\) and \(d_m\) via re-indexing. In this case, the convention dictates that the barcode be represented by half-open intervals. These intervals exclude the real-valued death time of the subsequent step of the filtration, in relation to when the feature was last designated ``alive'' using natural indices:
$$
\Pers(\H_*(X_\bullet)) = \big\{ [t_{b_m},\, t_{d_{m} + 1}) \big\}_{m = 1}^M.
$$
This convention also involves setting \(t_0 = -\infty\)---notice that the index \(i=0\) might appear in the barcodes of relative (co)homology---and \(t_{M+1} = \infty\). It is also customary to discard intervals where \(t_{b_m} = t_{d_{m}+1}\).

\begin{example}
Consider the filtration in Figure \ref{fig:filtration_triangle}, where our filtered complex is a triangle with vertices at \((0,0),\, (1,0)\), and \((\sqrt{3}/2,\sqrt{3}/2)\), where we add edges of increasing length and finally fill in the triangle.
\begin{figure}[H]
    \centering
    \includegraphics[width=0.8\columnwidth]{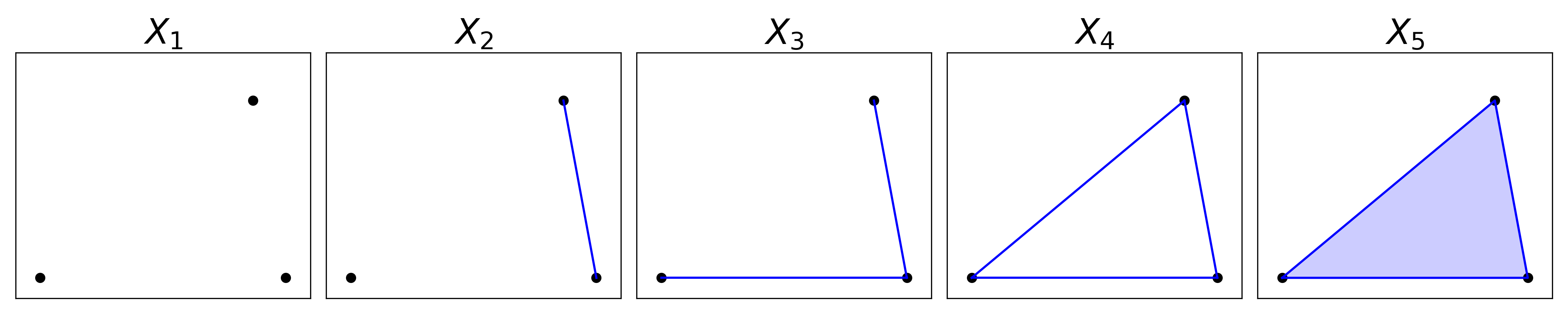}
    \caption{Filtration of a triangle indexed with natural numbers.}
    \label{fig:filtration_triangle}
\end{figure}
The barcodes for such a filtration for the \(0\)- and \(1\)-dimensional homologies are
\[\Pers(H_0(X_\bullet)) = \{[1,1],\, [1,2], [1,5]\}, \quad \Pers(H_1(X_\bullet)) = \{[4,4]\}.\]
Considering the re-indexing given by the diameter of the larger simplex in the complex $t_i=\max\{ \mathrm{diam}\,(\sigma): \sigma \in X_i\}$, the previous barcodes would then become
\[\Pers(H_0(X_\bullet)) = \{[0,\,0.88),\, [0,\,1), [0,\,+\infty)\}, \quad \Pers(H_1(X_\bullet)) = \{[1.22,\,1.22)\},\]
so that we may discard the \(1\)-dimensional persistent homology.
\end{example}

The four standard persistence modules carry the same information: the barcode of the setting of absolute (or relative) homology setting is the same as the barcode of the setting of absolute (or relative) cohomology. We can also find a bijection between the bars in the barcodes of the relative setting and the bars in the corresponding absolute setting. For further the details on this equivalence, see \cite{silva_dualities_2011}.

\subsection{Image-Persistence and Interval Matching}
\label{subsec:image_persistence}

The idea underlying image-persistence is to study the persistent homology of some filtered complex inside another larger filtered complex. Let \(X\) and \(Z\) be finite simplicial complexes and \( f : X \to Z\) an injective map between them. Let \(\{X_i : 1 \leq i \leq n\}\) and \(\{Z_i: 1 \leq i \leq n\}\) be filtrations associated to the previous complexes and denote the restrictions to the steps in the filtrations by
\[f_i := f\vert_{X_i} : X_i \to Z_i. \]
Note that these are also injective maps, which gives rise to the following commutative diagram for all \(1\leq i\leq n-1\):
\[\begin{CD}
		X_i @>\iota_i^X>> X_{i+1} \\
		@Vf_iVV @VVf_{i+1}V  \\
		Z_i @>\iota_i^Z>> Z_{i+1} \\
\end{CD}\]
where \(\iota_i^X\) and \(\iota_i^Z\) are the corresponding inclusion maps between consecutive steps in the corresponding filtration. 

Applying the homology functor to the previous diagram gives rise to another commutative diagram:
\begin{align*}
	\begin{CD}
		\H_*(X_i) @>\H_*(\iota_i^X)>> \H_*(X_{i+1}) \\
		@V\H_*(f_i)VV @VV\H_*(f_{i+1})V \qquad  \\
		\H_*(Z_i) @>\H_*(\iota_i^Z)>> \H_*(Z_{i+1}) \\
	\end{CD}
\end{align*}
which now involves the homology groups and the induced linear maps. The commutativity of this diagram allows for the following definition.

\begin{definition}[Image-persistent homology]
	The persistence module
	\[ \Im \H_* (f_\bullet) : \quad \mathrm{Im}(\H_*(f_i)) \to \mathrm{Im}(\H_*(f_{i+1}))   \]
	given by the subspaces \(\mathrm{Im}(\H_*(f_i))\subset \H_*(Z_i)\) and the restrictions of the maps \(\H_*(\iota_i^Z)\) is called \emph{image-persistent homology}.
\end{definition}

The elements in \(\Im(\H_*(f_i))\) can be seen as \textit{cycles in \(X_i\) up to boundaries in \(Z_i\)}, which we gain by studying one filtration inside another. \\

\begin{remark}
\label{rmk:relation_image_persistence_and_original_modules}
    Since \(\Im \H_*(f_i)\) is a subspace of \(\H_*(Z_i)\), a death in the image-persistence module implies a death in the persistent homology of the space \(Z\). At a topological level, this implication can be linked to the fact that every cycle from an image-persistence module is in fact a cycle in the codomain of the function used to define the image-persistence module. This cycle may have been born before as a cycle of the codomain, however, it gets bounded at the same time for both the image-persistence and the persistence module of the codomain. On the other hand, a birth in the image-persistence module implies a birth in the persistent homology of the space \(X\). This is a result from the fact that every cycle in the image persistence corresponds to a cycle of the domain, which may get bounded in the image-persistence module at an earlier time as it is studied within the ``larger'' codomain.  See \cite{cohen-steiner_persistent_2009} for further details on the relations between these three modules.
\end{remark}

The definition of interval matching introduced by \cite{reani_cycle_2021} is based on image-persistence to compare the persistence bars of two diagrams and is restricted to \emph{Morse filtrations}. 

\begin{definition}[Morse filtration] A filtration $\{X_t:t\in \R\}$ is a \emph{Morse filtration} if there exists a finite set $T= \{t_1,...,t_n\}\subset \R$ such that the following are satisfied:
	\begin{enumerate}
		\item For all $t\notin T$, there exists $\epsilon>0$ small enough such that for every $0<\epsilon'<\epsilon$ the map 
		$$ \H_*(i) : \H_*(X_{t-\epsilon'}) \to \H_*(X_{t+\epsilon'}) $$
		induced by inclusion is an isomorphism for every homology group. Equivalently, the homology does not change at $t$.
		\item For all $t\in T$, there exists $\epsilon>0$ small enough so that for any $0<\epsilon'<\epsilon$ either
		\begin{enumerate}
			\item $ \H_*(i) : \H_*(X_{t-\epsilon'}) \to \H_*(X_{t+\epsilon'}) $ is injective and the dimension of the vector space increases by one, or
			\item $ \H_*(i) : \H_*(X_{t-\epsilon'}) \to \H_k(X_{t+\epsilon'}) $ is surjective and the dimension decreases by one.
		\end{enumerate}
		Equivalently, the homology changes allowed are either the creation of a single new cycle or the termination of a single existing cycle.
	\end{enumerate}
	\label{def:morse_filtration}
\end{definition}

We now review how to match the persistence intervals of two filtered complexes inside a third comparison space. Let $X, Y, Z$ be finite simplicial complexes with Morse filtrations \(\{X_i: 1 \leq i \leq n\}\), \(\{Y_i: 1 \leq i \leq n\}\), and \(\{Z_i: 1 \leq i \leq n\}\). Assume we have injective maps
\[f_i : X_i \to Z_i, \qquad g_i: Y_i \to Z_i\]
for every \(1\leq i\leq n\) such that \(f_j\vert_{X_i} = f_i\) and \(g_j\vert_{Y_i} = g_i\) for every \(i\leq j\). With these assumptions, \cite{reani_cycle_2021} match persistence intervals as follows.

\begin{definition}[Matching intervals, \cite{reani_cycle_2021}]
	\label{def:cycle_matching_morse}
Let $\alpha \in \Pers(\H_*(X_\bullet))$ and $\beta \in \Pers(\H_*(Y_\bullet))$. The intervals $\alpha$ and $\beta$ are \emph{matching intervals via $Z_\bullet$} if there exist $\tilde{\alpha} \in \Pers(\Im \H_*(f_\bullet))$ and $\tilde{\beta} \in \Pers(\Im \H_*(g_\bullet))$ such that
\begin{align*}
\birth \alpha &= \birth \tilde{\alpha} \\
\birth \beta &= \birth \tilde{\beta} \\
\death \tilde \alpha &= \death \tilde{\beta}.
\end{align*}
\end{definition}

The intuition behind these criteria for matching intervals stems from Remark \ref{rmk:relation_image_persistence_and_original_modules}: every bar in the barcode of the image-persistence module \(\H_\ast(f_\bullet)\) arises from a bar in the barcode of the persistence module \(\H_\ast(X_\bullet)\) so that both share the same birth time; we have the same respective coincidence for the modules \(\H_\ast(g_\bullet)\) and \(\H_\ast (Y_\bullet)\). This justifies our procedure to first match a bar \( \alpha \in \Pers(\H_\ast(X_\bullet))\) and a bar \(\beta \in \Pers(\H_\ast(Y_\bullet))\) with the bars \( \tilde{\alpha} \in \Pers(\H_\ast(f_\bullet))\) and \(\tilde{\beta} \in \Pers(\H_\ast(g_\bullet))\) when their birth times coincide. Similarly, the death of a bar in an image-persistence module implies that there is a bar in the barcode of \(\H_\ast(Z_\bullet)\) sharing that same death time. Consequently, once we have identified the bars \(\tilde{\alpha}\) and \(\tilde{\beta}\), the rule to match them is that they both share the same death time, and thus are related to the same persistence interval of the module \(H_\ast(Z_\bullet)\).

\begin{remark}
	\label{rmk:morse_condition_matching}
	The Morse assumption is crucial in order for the notion in Definition \ref{def:cycle_matching_morse} to be well-defined. Having Morse filtrations for \(X\) and \(Y\) ensures that there is at most one birth at each time in \(\H_*(X_\bullet)\) and \(\H_*(Y_\bullet)\). From the definition of image-persistence, this also holds in the respective image-persistence modules. Recall that a birth in the image-persistence module means a birth in the corresponding persistent homology module of \(X\) or \(Y\). This allows each bar from the image-persistence module to have the same birth time as exactly one bar in the associated persistent homology. 
	
	From the death perspective, recall that a death happening in any of the image-persistence modules means a death in \(\H_*(Z_\bullet)\). Thus, there is also at most one bar in each image-persistence diagram dying at any given time. Consequently, each bar from an image-persistence module can share the same death time with at most one bar from the other image-persistence module.  These notions of uniqueness induced by assuming Morse filtrations guarantee that there are no ambiguous matchings.
\end{remark}

\subsection{Matching Affinity and Prevalence Score}
\label{sec:matching_and_prevalence}

The \emph{prevalence score} was proposed by \cite{reani_cycle_2021} as an alternative measure to persistence---in the sense of interval length---as an indicator for topological significance in noisy data.  It takes inspiration from bootstrapping techniques, which is a well-known and powerful subsampling with replacement method, originally proposed in the statistical literature by \cite{efron1982jackknife}. 

The formulation of the prevalence score takes into account the inherent tendency that as the sample size grows, noisy generators tend to reappear frequently.  Due to this tendency, the \emph{affinity} of a match must first be discussed before prevalence may be considered. Affinity is a score assigned to every match that considers the lifetimes of the persistent cycles and image-persistent cycles involved in the definition of interval matching. Recall that the \emph{Jaccard index} of two intervals \(I\) and \(J\) is given by
\[\mathrm{Jac}(I,J) : = \dfrac{\norm{I \cap J}}{\norm{I \cup J}},\]
where \(\norm{\cdot}\) denotes the length of an interval.
\begin{definition}[Matching affinity, \cite{reani_cycle_2021}]
	\label{def:affinity}
	The \emph{matching affinity} of two bars \(\alpha, \beta\) matching through their image-bars \(\tilde{\alpha}, \tilde{\beta}\) is defined as the product
	\[\rho(\alpha,\beta) := \mathrm{Jac}(\alpha,\beta) \cdot \mathrm{Jac}(\alpha, \tilde{\alpha})\cdot \mathrm{Jac}(\beta,\tilde{\beta}).\]
\end{definition}
With this definition, the prevalence score may now be formally introduced.
\begin{definition}[Prevalence score,  \cite{reani_cycle_2021}]
	\label{def:prevalence}
	Given some reference space \(X = X_{\mathrm{ref}}\) and resampling spaces \(X^{(1)},\ldots,X^{(K)}\) , any \(\alpha \in \Pers(H_*(X_\bullet))\) has a \emph{prevalence score} defined by
	\[\mathrm{prev}(\alpha):= \dfrac{1}{K} \sum_{k=1}^{K} \rho (\alpha, \beta_k(\alpha)),\]
	where \(\beta_k(\alpha)\) is the unique bar in \(X^{(k)}\), for \(1\leq k\leq K\), matched to \(\alpha\) using \(Z = X_{\mathrm{ref}}\cup X^{(k)}\) as a comparison space and the inclusions into the union as the connecting maps \(f: X_{\mathrm{ref}} \to Z\) and \(g: X^{(k)} \to Z\) (if a matching does not exist, set \(\rho = 0\)).
\end{definition}

\subsection{Clearing Algorithm and Cohomology}
\label{subsec:ripser}

The basic algorithm to compute the persistent homology of a filtered complex is based on reducing each column of the matrix of the boundary operator on the complex by adding columns on its left, from left to right, to obtain a reduced matrix.  From this reduced matrix, the barcode can be readily attained.

\cite{chen_persistent_2011} proposed an optimization of this process called the \emph{clearing algorithm}, based on the observation that since the matrix to reduce comes from a boundary operator, some columns in the reduced matrix must be null after the reduction and do not play a role in the reduction process. The clearing algorithm then reduces the boundary matrix in blocks from right to left so that it becomes possible to detect these null columns beforehand and set them directly to \(0\). In that way, it is possible to avoid reducing these columns and thereby accelerate the computation. However, the increase in speed is burdened by the large number of columns in the first block that must be reduced in the boundary matrix.

In an application to compute circular coordinates, \cite{de_silva_persistent_2011} observed that computing persistent cohomology was generally faster than computing persistent homology. This phenomenon was later confirmed by \cite{bauer_phat_2017}, who also realized that this increase in speed was coming from an implicit use of the clearing algorithm to compute persistent cohomology by  \cite{de_silva_persistent_2011}. One of the contributions of Ripser \citep{bauer_ripser_2021}, which also implements this optimization, is to provide a formal argument for this increase in speed: the advantage of using the clearing algorithm in persistent cohomology stems from the fact that in the relative coboundary matrix, the first block to reduce is significantly smaller.
From the reduced coboundary matrix, the barcode for the relative cohomology setting (\ref{eq:rc}) is read off, which is equivalent to the barcode for the homology setting (\ref{eq:ah}) as established by \cite{silva_dualities_2011}.

We note in particular that Ripser only considers Vietoris--Rips persistent homology, which is based on the Vietoris--Rips filtration, and does not compute other filtrations.  The Vietoris--Rips filtration is a standard filtration often considered in computational applications and settings.  
Recall that the Vietoris--Rips filtration of a finite metric space $(\mathcal{P},d)$ is $\VR_\bullet(\mathcal{P}, d)$ where the simplicial complex at filtration value $\epsilon$ is
\[\VR_\epsilon(\mathcal{P}, d) = \{ \emptyset \neq S \subset \mathcal{P} ~\vert~ \forall\ p,q \in S, \, d(p,q) \leq \epsilon \}.\]
By applying the homology functor, we obtain the Vietoris--Rips persistent homology $\H_*(\VR_\bullet(\mathcal{P}, d))$.

\section{Cycle Matching in the Setting of Cohomology}
\label{sec:cycle_matching_cohomology}

In this section, we present our theoretical contributions.  Specifically, we address the current gaps in the literature by providing a comprehensive account of the extension of image-persistence to the outstanding settings of the four standard persistence modules and the relations between such modules.
Subsequently, we specialize the definition of interval matching to simplex-wise filtrations and outline how to implement our specialization using Ripser-image. We finish the section with a case study of alternative definitions of the matching affinity.

\subsection{The Four Image-Persistence Modules}

We only need functoriality applied to the following commutative diagram 
\[\begin{CD}
		X_i @>\iota_i^X>> X_{i+1} \\
		@Vf_iVV @VVf_{i+1}V  \\
		Z_i @>\iota_i^Z>> Z_{i+1} \\
\end{CD}\]
for \(1\leq i\leq n-1\) to define image-persistence, which can then be easily extended to obtain four image-persistence modules as parallels to the four standard persistence modules presented previously in Section \ref{subsec:standard_persistence_modules}. 

In the setting of \textit{absolute cohomology}, applying the corresponding homology functor gives us the following commutative diagram
\[\begin{CD}
	\H^*(X_i) @<\H^*(\iota_i^X)<< \H^*(X_{i+1}) \\
	@A\H^*(f_i)AA @AA\H^*(f_{i+1})A \\
	\H^*(Z_i) @<\H^*(\iota_i^Z)<< \H^*(Z_{i+1}) \\
\end{CD} 
\]
for every \(1\leq i \leq n-1\). Since we are working with field coefficients, the objects with superscripts are dual to the objects with subscripts from the diagram for homology. The commutativity of the diagram allows for the following definition.

\begin{definition}[Image-persistent cohomology]
	The \emph{image-persistent cohomology} is defined as the persistence module
	\[ \Im \H^* (f_\bullet) : \quad \mathrm{Im}(\H^*(f_i)) \to \mathrm{Im}(\H^*(f_{i-1}))   \]
	given by the subspaces \(\mathrm{Im}(\H^*(f_i)) \subset \H^*(X_i)\) and the restrictions of the maps \(\H^*(\iota_i^X)\).
\end{definition} 

We now consider the relative settings. Recall that we have a map \(f: X \to Z\) such that \(f(X_i) \subset Z_i\); we  denote this by
\[ f : (X, X_i) \to (Z, Z_i)\]
for every \(1 \leq i \leq n\). Since we also have \(X_i \subset X_{j} \subset X\), for \(i\leq j\), we can write
\[ \iota^X : (X,X_i) \to (X, X_{j}), \quad \ i\leq j\]
for the identity in \(X\). The same can be written for the identity \(\iota^Z\)
\[ \iota^Z : (Z,Z_i) \to (Z, Z_{j}), \quad \ i\leq j.\]
This gives us the following commutative diagram
\[\begin{CD}
	(X,X_i) @>\iota^X>> (X,X_{i+1}) \\
	@VfVV @VVfV  \\
	(Z,Z_i) @>\iota^Z>> (Z,Z_{i+1}) \\
\end{CD}\]
for \(1 \leq i \leq n-1\). Using functoriality in the relative setting we obtain the subsequent commutative diagrams
\[\begin{CD}
	\H_*(X,X_{i}) @>\H_*(\iota^X)>> \H_*(X,X_{i+1}) \\
	@V\H_*(f, f_i) VV @VV\H_*(f, f_{i+1})V  \\
	\H_* (Z,Z_{i}) @>\H_*(\iota^Z)>> \H_*(Z,Z_{i+1}) \\
\end{CD}\hspace{16pt}, \hspace{25pt}
\begin{CD}
	\H^*(X,X_{i}) @<\H^*(\iota^X)<< \H^*(X,X_{i+1}) \\
	@A\H^*(f, f_i) AA @AA\H^*(f, f_{i+1})A  \\
	\H^* (Z,Z_{i}) @<\H^*(\iota^Z)<< \H^*(Z,Z_{i+1}) \\
\end{CD}\]
where again, the subscripts and superscripts mean duality. Observe the notation of the homology functor applied to \(f : (X, X_i) \to (Z, Z_i)\). Commutativity again allows for the following definitions.

\begin{definition}[Image-persistent relative homology]
	The \emph{image-persistent relative homology} is the persistence module
	\[\Im \H_* (f, f_\bullet) : \quad \mathrm{Im}(\H_*(f, f_i)) \to \mathrm{Im}(\H_*(f, f_{i+1}))   \]
	given by the vector spaces \(\mathrm{Im}(\H_*(f, f_i)) \subset \H_*(Z,Z_i)\) and the restrictions of the linear maps \(\H_*(\iota^Z)\).
\end{definition}

\begin{definition}[Image-persistent relative cohomology]
	The \emph{image-persistent relative cohomology} is the persistence module
	\[\Im \H^*(f, f_\bullet) : \quad \mathrm{Im}(\H^*(f, f_i)) \to \mathrm{Im}(\H^*(f,f_{i-1}))  \]
	given by the vector spaces \(\mathrm{Im}(\H^*(f,f_i)) \subset \H^*(X,X_i)\) and the restrictions of the linear maps \(\H^*(\iota^X)\).
\end{definition}

\begin{remark}
    The four image-persistence modules above currently appear in \cite{bauer_lifespan_2021} as an example of application of a broader theory of \emph{lifespan functors}, i.e., endofunctors in the category of persistence modules and \emph{matching diagrams}---equivalent to barcodes---which are related to boundedness properties of the intervals. Their work is precisely motivated by the extension of the duality results by \cite{silva_dualities_2011} to the setting of image-persistence. A version of a result from \cite{bauer_lifespan_2021} is then revisited in \cite{bauer_efficient_2022} in Proposition 3.12 in order to implement Ripser-image, which plays a central role in our work. This result will be further explained in Section \ref{sec:equivalence_image_pers_settings} and is fully referenced in Proposition \ref{prop:bauer_schmahl}.
\end{remark}

\subsection{Equivalence Among the Four Image-Persistence Settings}
\label{sec:equivalence_image_pers_settings}

A natural question to ask after introducing the four image-persistence modules is whether we can expect equivalences among them akin to the ones proved by \cite{silva_dualities_2011} for the standard persistence modules. In search of a first immediate answer to this question, we check directly whether the persistence modules of homology and cohomology provide the same information. 
\begin{proposition}
	The following equalities hold:
	\begin{align*}
		\Pers(\Im \H_*(f_\bullet)) &= \Pers(\Im \H^*(f_\bullet)), \\
		\Pers(\Im \H_*(f, f_\bullet)) &= \Pers(\Im \H^*(f,f_\bullet)).
	\end{align*}
\end{proposition}
\begin{proof}
It is sufficient to prove that the maps naturally induced between the images,
$$ {\H_*(\iota_{i,j}^Z)}\vert_{ \mathrm{Im} (\H_*(f_i))} :  \mathrm{Im} (\H_*(f_i))\rightarrow \mathrm{Im} (\H_*(f_{j})) $$
and
$$ {\H^*(\iota_{i,j}^X)}\vert_{ \mathrm{Im} (\H^*(f_{j})) } : \mathrm{Im} (\H^*(f_{j})) \rightarrow \mathrm{Im} (\H^* (f_i)),$$
for all \(i \leq j\) have the same rank. This is true since 
\begin{align}
\rank {\H_*(\iota_{i,j}^Z)}\vert_{ \mathrm{Im} (\H_*(f_i))} &= \rank \left( \H_* (\iota_{i,j}^Z) \circ \H_*(f_i) \right) \nonumber\\
&= \rank \left(\H^* (f_i) \circ \H^*(\iota_{i,j}^Z)\right) \label{eq:rk2} \\
&= \rank \left(\H^* (\iota_{i,j}^X) \circ \H^*(f_{j})\right) \label{eq:rk3} \\
&= \rank {\H^*(\iota_{i,j}^X)}\vert_{ \mathrm{Im} (\H^*(f_{j})) } \nonumber
\end{align}
where the second equality \eqref{eq:rk2} is given by duality and the third equality \eqref{eq:rk3} is given by the commutativity of the diagram in absolute cohomology.
Since persistence barcodes are uniquely
determined by dimensions and ranks, we have shown that the image-persistent absolute homology and image-persistent cohomology barcodes are the same. The same proof applies to prove equality of the barcodes in the relative setting.
\end{proof}

As for equivalence between the absolute and relative settings, the arguments used in \cite{silva_dualities_2011} are not directly applicable for image-persistence. However, if \(\mathrm{Pers}_0\) denotes the finite intervals and \(\mathrm{Pers}_\infty\) the infinite intervals of a given persistence barcode,
we have the following correspondence. 

\begin{proposition}[\cite{bauer_efficient_2022}, Proposition 3.12]
    We have
    \[\Pers_0 (\Im \H_*(f_\bullet)) = \Pers_0 (\Im \H^{\ast +1}(f,f_\bullet)).\]
    Additionally, the map \(I \to T \setminus I\) defines the bijections
    \[\begin{array}{rcl}
        \Pers_\infty (\Im \H_*(f_\bullet)) & \cong & \Pers_\infty (\H^*(X,X_\bullet)),\\[2pt]
        \Pers_\infty (\Im \H^*(f,f_\bullet)) & \cong & \Pers_\infty( \H_*(Z_\bullet)).
    \end{array}\]
    \label{prop:bauer_schmahl}
\end{proposition}

This result means that in order to determine the barcode of \(\Im \H_*(f_\bullet)\), it suffices to compute $\Pers_\infty(\H^*(X,X_\bullet))$ and \(\Pers_0(\Im \H^*(f, f_\bullet))\). Both of these persistence diagrams may be computed applying a matrix reduction algorithm to appropriate boundary matrices. \cite{bauer_efficient_2022} also show that the clearing algorithm implemented in Ripser (see Section \ref{subsec:ripser}) can be applied to compute image-persistence. In this way, the code for Ripser can be fully adapted to this setting to achieve state-of-the-art computations for image-persistence. 

\subsection{Matching Intervals in Non-Morse Filtrations}

Ripser-image provides the barcode of the image-persistent homology for Vietoris--Rips filtrations. As noted previously in Remark \ref{rmk:morse_condition_matching}, this presents a significant obstacle to implement efficient interval matching using Ripser-image: Vietoris--Rips filtrations are not Morse filtrations. To overcome this limitation, we introduce a specialization of Definition \ref{def:cycle_matching_morse} that resolves the matches between bars with shared birth or death time. We first recall the definition of simplex-wise filtration. 

\begin{definition}
\label{def:simplexwise_filtration}
	A filtered complex \(\{X_i : i \in I\}\) is \emph{essential} if $i\neq j$ implies $X_i \neq X_j$. Additionally, it is a \emph{simplex-wise filtration} if for every $i\in I$ such that $X_i \neq \emptyset$ there is some simplex $\sigma_{i}$ and some index $j<i$, such that $X_i \smallsetminus X_j = \{\sigma_i\}$.
\end{definition}
Observe that in simplex-wise filtrations, there is a bijection between the indices of the filtration and the simplices of the complex \(X\).    
Consequently, the persistence pairs in the intervals of the barcode can be associated with particular simplices. We call the simplices corresponding to birth times \emph{positive simplices} and those corresponding to death times \emph{negative simplices}; we say that an interval is \emph{created} by its positive simplex and \emph{destroyed} by the negative simplex.

In addition, simplex-wise filtrations are Morse filtrations, which allows us to directly apply the definition of interval matching proposed by \cite{reani_cycle_2021} (Definition \ref{def:cycle_matching_morse}): given the correspondence between birth (resp.~death) times and positive (resp.~negative) simplices, we can rephrase Definition \ref{def:cycle_matching_morse} in the following manner.
Let $X, Y, Z$ be finite simplicial complexes with \emph{simplex-wise} filtrations \(\{X_i: i \in I\}\), \(\{Y_i: i \in I\}\), and \(\{Z_i:i \in I\}\). Assume we have injective maps \(f:X \to Z\) and \(g: Y \to Z\) with the usual notation for the restrictions
\[f_i : X_i \to Z_i, \qquad g_i: Y_i \to Z_i,\]
for every \(i \in I\).

\begin{definition}[Interval matching for simplex-wise filtrations]
	\label{def:general_cycle_matching}
	Let $\alpha \in \Pers(\H_*(X_\bullet))$ and $\beta \in \Pers(\H_*(Y_\bullet))$. The intervals $\alpha$ and $\beta$ are \emph{matching intervals via $Z_\bullet$}, if there exist $\tilde{\alpha} \in \Pers(\Im \H_*(f_\bullet))$ and $\tilde{\beta} \in \Pers(\Im \H_*( g_\bullet))$ such that the following conditions are satisfied:
	\begin{itemize}
		\item \(\alpha\) and \(\tilde{\alpha}\) are created by the same simplex (seen in \(X\) and in \(f(X)\), respectively);
  		\item \(\beta\) and \(\tilde{\beta}\) are created by the same simplex (seen in \(Y\) and in \(g(Y)\), respectively);
    	\item \(\tilde{\alpha}\) and \(\tilde{\beta}\) are destroyed by the same simplex in \(Z\).
	\end{itemize}
\end{definition}

In Definition \ref{def:general_cycle_matching}, the use of the phrase ``by the same simplex'' in relation to intervals \(\alpha\) and \(\tilde{\alpha}\) means that if \(\sigma\) is the positive simplex associated to \(\alpha\), then \(f(\sigma)\) is the positive simplex associated to \(\tilde{\alpha}\); this meaning also applies to the bars \(\beta\) and \(\tilde{\beta}\) with their positive simplices connected by the function \(g\). Notice that this association is well defined since both \(f\) and \(g\) are assumed to be injective. Similarly, the notion of \(\tilde{\alpha}\) and \(\tilde{\beta}\) being destroyed \emph{by the same simplex} means that the negative simplices associated to these intervals are precisely represented by a single simplex within the larger complex \(Z\).

Again, many constructions of filtered complexes do not comply with Definition \ref{def:simplexwise_filtration}, in particular, Vietoris--Rips filtrations are not simplex-wise in general. However, for any filtered complex we can always find a re-indexing that refines the filtration and turns it into an essential simplex-wise filtration. This can be done by finding a partial ordering of the simplices in each of the steps of the original filtration that extends to a total ordering in the whole complex. For instance, Ripser uses a lexicographic refinement of the Vietoris--Rips filtration, which orders the simplices by dimension, diameter and a combinatorial system (see \cite{bauer_ripser_2021} for further detail). 

Consequently, restricting the matching to simplex-wise filtrations does not introduce any additional difficulties; rather, it enables the application of cycle matching to general filtrations. The process described above is a standard procedure in persistent homology solvers (i.e., algorithms that compute persistent homology barcodes), which compute the barcode for the refined simplex-wised filtration and then retrieve the barcode of the original filtration by relabelling the endpoints of the intervals as explained in Section \ref{subsec:standard_persistence_modules}. This process can be extended to interval matching in a similar manner: given a general filtration, we refine it to a simplex-wise filtration, compute the interval matching following Definition \ref{def:general_cycle_matching}, and then recover the interval matching for the original filtration. This approach addresses a gap in the work of \cite{reani_cycle_2021} and thus makes their ideas more general and widely applicable, increasing their practical utility.

\begin{remark}
    Notice that to implement the interval matching in Definition \ref{def:general_cycle_matching} for general filtrations, we need to further assume that underlying simplex-wise refinements are compatible in the three filtrations. This means that the simplices are added to the persistence modules \(\H_*(X_\bullet)\) and \(\H_*(Y_\bullet)\) and to their image-persistence modules \(\Im \H_*(f_\bullet)\) and \(\Im \H_*(g_\bullet)\) in the same order. This can always be achieved by first setting the orders in \(X\) and \(Y\) and then in \(Z\) accordingly.
\end{remark}

\subsection{Implementing Cycle Matching with Ripser-Image}
\label{subsec:implementation}

The input to Ripser-image consists of two Vietoris--Rips filtrations 
\[X_\bullet = \VR_\bullet(\mathcal{P}, d) \quad \text{and} \quad Z_\bullet = \VR_\bullet(\mathcal{P}, d'),\]
where the two metrics \(d, d'\) on a finite set \(\mathcal{P}\) satisfy \(d(p,q) \geq d'(p,q)\) for all \(p, q \in \mathcal{P}\). However, the setting for interval matching is slightly different. 

From section \ref{subsec:image_persistence}, to implement interval matching on finite point clouds \(\mathcal{X}\) and \(\mathcal{Y}\) living in the same ambient space, we can consider their union \(\mathcal{P} = \mathcal{X} \cup \mathcal{Y}\) and any metric \(d'\) on \(\mathcal{P}\) induced from that ambient space. This will induce metrics \(d_X := d'\vert_{\mathcal{X}} \) and \(d_Y :=d'\vert_{\mathcal{Y}}\) on the smaller point clouds as well. Consider the extension
\[(\mathcal{X}, d_X) \subset (\mathcal{P}, d_X')\]
such that the metric \(d_X'\) is obtained by setting a very large distance between any point in \(\mathcal{X}\) and any point in \(\mathcal{Y}\), and any pair of points in \(\mathcal{Y}\), all seen in the union. Then, up to a threshold corresponding to that large distance and up to the points in \(\mathcal{P}\setminus \mathcal{X} = \mathcal{Y}\), we have
\[X_\bullet = \VR_\bullet (\mathcal{X}, d_X) \simeq \VR_\bullet (\mathcal{P}, d_X')\]
which puts us in the setting of Ripser-image. Notice that in the matrix representation of \(d_X\) there are rows and columns corresponding to points in \(\mathcal{Y}\). This construction can be also applied to \((\mathcal{Y}, d_Y)\), preserving the same order as before in rows and columns to ensure compatibility. In this manner, we obtain the three Vietoris--Rips filtrations
\[X_\bullet \subset Z_\bullet \supset Y_\bullet,\]
and we consider the inclusions as the connecting functions for the matching.

The code for Ripser and Ripser-image assigns a unique index to any simplex of the input filtered complex using a lexicographic refinement. However, it does not provide the indices associated to the positive and negative simplices of the persistence intervals of the barcode in its output. These values can be readily retrieved with a slight change the original code. By arranging the matrices representing the finite metric spaces as explained above, these indices allow for the implementation of Definition \ref{def:general_cycle_matching} without affecting the computational runtime of both programs. 

It is crucial to observe at this point that a change in the order of the columns and rows within the distance matrices used as input for Ripser-image, which are in correspondence with the points of \(\mathcal{X}\) and \(\mathcal{Y}\) in the setting described above, will alter the indices assigned to simplices through the lexicographic refinement. This could, in principle, alter the outcome of the interval matching using Ripser-image as described thus far. Nonetheless, as long as we are consistent with the ordering in all of the three matrices involved in the computations, this observation poses no problem for the implementation of the interval matching. In our code, the points in \(\mathcal{X}\) occupy the foremost positions, preserving the same order in all three matrices, and the points in \(\mathcal{Y}\) assume the terminal positions, also preserving their order across matrices.

\paragraph{A Note on Terminology: Cycle Matching.} 

\cite{reani_cycle_2021} refer to this framework of matching intervals in persistent homology as \textit{cycle registration}. In this paper, whenever we use the term ``cycle matching,'' we assume that there is a way of finding cycles in the final simplicial complex of the filtration that correspond to the intervals in its barcode. Note that, in general, these representative cycles are not unique---in fact, the persistence intervals are associated to homology classes, i.e., equivalence classes of cycles. However, there are methods to uniquely determine the representative cycles. One such method considers the columns of the reduced boundary matrix corresponding to the killing simplices, which provides the simplices that make up a cycle killed by that simplex.

Further on in section \ref{sec:applications}, we implement cycle matching by using the version \texttt{ripser-tight-representative-cycles} of Ripser. This feature provides the representative cycles corresponding to the intervals in the barcode computed by Ripser. Note that Ripser does not reduce the boundary matrix but the relative coboundary matrix (see the previous discussion from section \ref{subsec:ripser}), and thus we cannot implement the aforementioned method directly. However, \cite{cufar_fast_2021} develop an adaptation of this idea to obtain state-of-the-art computations of barcodes and representatives. The method proposed in \cite{cufar_fast_2021} uses persistent cohomology to obtain the persistence pairs and
reduces the boundary matrix only using the columns corresponding to death indices. This is the technique that \texttt{ripser-tight-representative-cycles} implements.

\paragraph{Processing Multiple Jobs in Parallel.}

A computational advantage of the bootstrapping approach proposed by \cite{reani_cycle_2021}
is that this technique is parallelizable, despite the inherently non-parallelizable nature of persistent homology computations. Recall that once the barcode of the reference sample is computed, to obtain the prevalence score of its intervals, these intervals are matched with the intervals in the barcodes of \(K\)-many resamplings. These matchings can be processed in parallel jobs using a high performance computer cluster (HPC) with a workload manager and job scheduling system, such as SLURM or OpenPBS. In each of these jobs, first, the barcode of the corresponding resampling and the barcodes of the image-persistence modules involved are computed, then the generalized interval matching is implemented. This allows for a dramatic increase in efficiency in the computational runtime, with respect to a sequential execution of the code: 
the total computational time corresponds to the one of the slowest job,
instead of the sum of the computational times of the individual jobs.

\subsection{Revisiting Matching Affinity}
\label{sec:revisit_affinity}

The matching affinity introduced in Definition \ref{def:affinity} relies on a particular choice of pairs of intervals to compare through their Jaccard index. In principle, other selections are also valid to obtain different definitions of the matching affinity. We now study the behavior of four such affinities in the example of two circles with same radius but diverging centers. This will allow us to conclude that only one of these definitions exhibits a significant difference with respect to the others. 

From now on, we refer to the matching affinity of Definition \ref{def:affinity} as \emph{matching affinity $A$}
	\[\rho_A(\alpha,\beta) := \mathrm{Jac}(\alpha,\beta) \cdot \mathrm{Jac}(\alpha, \tilde{\alpha})\cdot \mathrm{Jac}(\beta,\tilde{\beta}),\]
where \(\alpha, \beta\) denote two bars matched through their image-bars \(\tilde{\alpha}, \tilde{\beta}\). This score involves the comparison of \(\alpha\) and \(\beta\) but also of each bar with its corresponding image-bar. Considering multiple ways to compare persistence bars and image-bars, we also have:
\begin{itemize}
    \item the \emph{matching affinity $B$} as \[\rho_B(\alpha,\beta) :=  \mathrm{Jac}(\Tilde{\alpha},\Tilde{\beta}) \cdot \mathrm{Jac}(\alpha, \tilde{\alpha})\cdot \mathrm{Jac}(\beta,\tilde{\beta});\]
    \item the \emph{matching affinity $C$} as \[\rho_C(\alpha,\beta) := \mathrm{Jac}(\alpha,\beta) \cdot\mathrm{Jac}(\Tilde{\alpha},\Tilde{\beta}) \cdot \mathrm{Jac}(\alpha, \tilde{\alpha})\cdot \mathrm{Jac}(\beta,\tilde{\beta}),\]
    \item the \emph{matching affinity $D$} as \[\rho_D(\alpha,\beta) := \mathrm{Jac}(\alpha,\beta) \cdot \mathrm{Jac}(\tilde{\alpha}, \tilde{\beta}).\]
\end{itemize}

The following concrete example provides an intuition on how the different affinities behave. Consider two circles of radius \(1\) and centers shifted by a distance \(s\). We expect that the matching affinity decreases as the center-to-center distance \(s\) increases, until reaching $0$ (i.e., no match) beyond a certain value. The result of this experiment is displayed in Figure \ref{fig:comparison_affinities}, where we see that all matching affinities decrease with respect to $s$ and that the cutoff value is $1$. Affinities $A, B,$ and $C$ follow very similar decreasing behaviors in a linear fashion, whereas affinity $D$ has a distinct plateau-like behavior. We now further investigate this phenomenon.

Assume that we have two persistence bars \(\alpha\) and \(\beta\) matched via their image-bars \(\tilde{\alpha}\) and \(\tilde{\beta}\). We know that the birth times of \(\alpha\) and \(\tilde{\alpha}\), and \(\beta\) and \(\tilde{\beta}\) coincide, respectively, and that the bars \(\tilde{\alpha}\) and \(\tilde{\beta}\) also have the same death time. Thus, having high affinity $A$, for instance, depends on the two following phenomena. Firstly, the bars \(\alpha\) and \(\beta\) must have similar birth and death times. This means that the cycles in \(X\) and \(Y\) that generated them should be similar in size. Secondly, the death time of the image-bars should also be similar to the death time of the original bars. Geometrically, this means that the cycles that generated the bars should have high overlapping surface when considered in the union of the point clouds. These ideas are illustrated in Figure \ref{fig:3_matches}, where the affinity $A$ of a match of two circles decreases significantly when the circles have different centers or radii.

Returning to Figure \ref{fig:comparison_affinities} with these ideas in mind, we see that there are few noticeable but not fundamental differences between the affinities $A, B,$ and $C$. Indeed, the Jaccard indices \(\mathrm{Jac}(\alpha, \beta)\) and \(\mathrm{Jac}(\tilde{\alpha}, \tilde{\beta})\) have similar magnitude---if a pair of cycles are similar in size in the spaces \(X\) and \(Y\), the cycles corresponding to their image-bars in the union will also have similar sizes. This implies that the affinities $A$ and $B$ are very similar. Affinity \(C\) is slightly lower than affinities \(A\) and \(B\) only because it has an extra multiplicative factor with magnitude less than one. 

It also makes sense that the matching affinities $A, B,$ and $C$ drop when the spatial overlap between cycles decreases. This happens because these affinities include the Jaccard indices between bars and image-bars, which are influenced by this overlap. The matching affinity \(D\) does not consider such a comparison, and thus, remains at a higher value until there is no match anymore, when it drops abruptly. Such a behavior could be useful in certain situations. In real-life applications, it might happen that we are interested in matching topological features that shrink or enlarge significantly, or which appear misplaced in the samples. Matching affinity $D$ would then be more sensitive to these matches, by assigning a higher prevalence score to them. However, this feature could be undesirable in other contexts, as we discuss in the next observation.

We now need to check whether the four affinities are consistent with the original motivation, which is the condition proposed in \cite{reani_cycle_2021}. Namely, random cycles that appear in resamplings and get matched several times should be assigned a low prevalence score. Similar to \cite{reani_cycle_2021}, we consider a uniform sampling of the unit square with \(N_{\mathrm{ref}} = 1000\) points and compute the prevalence scores of its bars by finding matches with \(K = 20\) different resamplings of \(N =1000\) points from that same distribution. The results of this experiment are given in Figure \ref{fig:random_cycles}. For affinity $D$, some random cycles are assigned quite high prevalence scores in the range 0.6--0.7. This must be taken into account when interpreting the affinity scores for applications using affinity $D$.

\section{Applications}
\label{sec:applications}

In this section, we demonstrate with numerous examples that cycle matching can be applied to real-life, large-scale, complex datasets from biology and astrophysics. 
Several applications motivate the usage of cycle matching and prevalence.
First, we can identify common topological features shared by two spaces as a direct application of cycle matching.
We can use this to track features both spatially and over time, on consecutive slices of an object or on consecutive time frames.
We demonstrate both of these applications in this section.

Second, the most prevalent features in data can be identifed by applying cycle matching repeatedly after resampling from the same distribution.  By doing this, we can detect prevalent cycles in large-scale  and complex data.  Here prevalence is computed using Affinity A exclusively. 
 We demonstrate this on cosmic web data and cell actin network data.
Prevalence gives rise to an enriched visualization via the \emph{prevalence-augmented barcode}, where length corresponds to persistence while thickness and color both correspond to prevalence.

\paragraph{Comparison with Classical Computer Vision Tools for Feature Tracking.} It is important to note at this point that the application we are suggesting here differs from existing algorithms for tracking features in computer vision, such as those implemented in OpenCV \citep{opencv_library}. These algorithms take images of an object of interest as input and use different methods to track the object through subsequent frames, identifying an area in the frame in which the object is located. The technique we are proposing fundamentally differs from the procedure above in four aspects. First, we track \emph{topological features}, which are structures that are inherently different from objects and object locations in images. 
Moreover, we do not need any prior knowledge of what to track: persistent homology directly detects the topological features of the image, and our only input is the video or stack of images where the topological features are present. 
In addition, our output is the chain of features from the persistent homology of the images that are matched in subsequent frames, which we can represent through cycle representatives as explained above, but which is not in principle related to any specific area in the image, as is the case of the output of the feature tracking algorithms.
Lastly, notice that we can track several features concurrently, without needing to re-run the whole algorithm, which is what usually happens in feature tracking algorithms in computer vision.

\subsection{Tunneling: Tracking Intervals Over Slices}

\label{sec:track_slices}

As a first application of cycle matching, we tracked intervals over two-dimensional slices of three-dimensional objects.
In biomedical imaging, for instance, it is common that data are made up of spherical or tubular elements, such as in vessels and other biological organs with channeling functions. Using cycle matching, we can match the closed contours delimiting the spherical or tubular elements across slices.

\paragraph{Data: Lateral Line in Zebrafish.}

To demonstrate this application, we used a biological imaging dataset, in particular the dataset with image ID 9836972 provided in \cite{10.7554/eLife.55913}. This is a stack of two-dimensional confocal images from the zebrafish posterior lateral line primordium (pLLP). The pLLP is a primitive expression of the lateral line---an organ in fish that allows them to detect the pattern of water flow over their body surface. It appears at the embryonic stage in the form of a rosette-shaped cluster of cells. The circular contours that we can see in the images of the dataset (see Figure \ref{fig:2D_slices}) are precisely these cells; we will track these along the height of the stack.

We considered a stack of \(15\) images with a \(0.66\ \mu m\) gap and size of \(300 \times 300\) pixels in each image, with a resolution of 0.1 \(\mu m\) per pixel. We thresholded the images with the Otsu method and uniformly sampled them with \(N = 1000\) points. We matched persistent intervals between pairs of Vietoris--Rips filtrations on consecutive slices. Some features were matched on consecutive frames and formed a tunnel: we were able to detect $32$ such tunnels, each stained with a different color. We computed the geometric generators drawn here to represent the persistence intervals that were matched using the \texttt{ripser-tight-representative-cycles} module of Ripser \citep{bauer_ripser_2021}. The results are shown in Figure \ref{fig:2D_slices}. As a byproduct of this approach, we can identify slices on which a cell appears then disappears.

\subsection{Video Data: Tracking Features Over Time}

\label{sec:track_video}

Cycle matching can be used to track topological features over time, by matching the barcodes of consecutive frames in a video, or, in a biological context, at different stages of disease development. This method detects common topological patterns surviving across consecutive time points and quantifies the quality of the match through the affinity scores.

\paragraph{Data: Heart Valves in Zebrafish.}

To illustrate this application, we analyzed a video of the atrioventricular valve (AVV) of a wild-type AB zebrafish, from \cite{scherz_developing_hearts_2008}.
This video is taken $76$ hours post fecundation and at a rate of $50$ milliseconds.
The specimen studied comes from a transgenic line that allows the monitoring of the two chambers that make up the primitive heart of the zebrafish embryo, and how its contraction over time generates embryonic heartbeats. The contraction is especially pronounced for the right chamber, as can be seen from Figure \ref{fig:heartbeat}.

We selected $10$ frames capturing one contraction and matched cycles on consecutive frames (Figure \ref{fig:heartbeat}). We sampled \( N = 500\) points on each of the images after applying a thresholding technique based on the mean of gray-scale values.
We successfully detected the persistent intervals delimiting the two chambers, and tracked them on all consecutive frames. We also tracked their size variation. Note that the matching affinities are much more variable for the cycle on the right (in red), which is expected since the right chamber changes abruptly in shape. As before, we show on each frame of Figure \ref{fig:heartbeat} the generator associated to each persistence interval, obtained with the ripser-tight-representative-cycles module of Ripser, while using the same color to stain matched cycles.

\paragraph{Data: Time-Lapse Images of Human Embryos.}

As another example to demonstrate feature tracking over time, we tracked intervals on $10$ consecutive frames of time-lapse embryo data from \cite{GOMEZ2022108258} (Figure \ref{fig:embryogenesis}). A time-lapse imaging (TLI) system with a special camera captures images of a human embryo every 50 to 100 minutes and features different stages of cellular division. We matched samples of \(N =500\) points on the images after applying a Sato operator and a threshold using the Otsu method. We were able in particular to detect cell division as the appearance of a new topological feature (in red).

\subsection{Prevalent Cycles}

\label{sec:appli_prevalent}

Our next application is to find prevalent cycles in order to reveal significantly organized topological patterns in data.  We will demonstrate this on cosmic web data (Figure \ref{fig:cosmic}) and cell actin data. 
Recall that this consists of comparing multiple resamplings $X^{(1)},\ldots,X^{(K)}$ to a reference space $X_{\rm ref}$ and finding all possible matching pairs of persistence intervals between $X_{\rm ref}$ and any $X^{(k)}$ for \(1\leq k\leq K\). 

This approach becomes especially useful in situations where the initial (unknown) distribution has high topological complexity, and for which we only have access to point cloud samplings. In other situations, the distribution may be partially or even entirely known already---for instance if we are given an image from which to sample points. Oftentimes, an image may suffer from noise or contrast variations, but we can recover the true cycles of the object. Even in the absence of noise, brighter and weaker signals may still capture interesting information (for instance, depth in 2D images)---prevalence takes this into account. We can also study the profile of prevalence scores corresponding to an image to characterize its topological structure in comparison to another image. We may visualize this using prevalence-augmented persistence barcodes, where the length of a bar still represents persistence in the usual sense, while its thickness and color represents prevalence.

\paragraph{Data: The Cosmic Web.}

First, we identified prevalent cycles in the cosmic web, based on the point distribution of galaxies from the BOSS CMASS database \citep{dawson_baryon_2013}. Matter in the universe is arranged along an intricate pattern involving filamentary structures  \citep{de_lapparent_slice_1986, york_sloan_2000}. However, the reconstruction of these filaments is still challenging; multiple methods to address this reconstruction problem have been proposed \citep{malavasi_characterising_2020}. Instead of detecting filaments with an uncertainty score \citep[e.g.,][]{duque_novel_2022}, we propose to detect cycles with a prevalence score.

The final version of the BOSS CMASS dataset used here is included in the current data release SDSS DR17 \citep{abdurrouf_seventeenth_2022} and was released in SDSS DR12 \citep{alam_eleventh_2015}. We selected galaxies with right ascension $170 < \mathrm{RA} < 190$, declination $30 < \mathrm{dec} < 50$ and redshift range $0.564 < z < 0.57$ and projected the points onto the $(\mathrm{RA},\, \mathrm{dec})$ 2D space.

We sampled the reference space $X_{\mathrm{ref}}$ using $N_\mathrm{ref} = 1000$ points and resampled the dataset $K = 20$ times with $300$ points in each resampling $X^{(k)}$, by adding Gaussian noise of magnitude $0.1$ to each point. The noise scale, required for the prevalence score bootstrapping technique, is determined based on both the data magnitude and the width of the filaments present in the original sample that we want to detect. We performed $20$ comparisons of barcodes to find matching cycles.
The results are shown in Figure \ref{fig:cosmic}, where cycles with different prevalence scores can be visualized.

\paragraph{Data: Cell Actin.}

Next, we computed the prevalence of cycles in biological imaging data of cell actin. Actin networks are essential in scaffolding the inner structure of cells, enabling in particular cell motility and reshaping. Figure \ref{fig:actin_data}, featuring data from \cite{svitkina_arp23_1999}, shows significant loss of actin filaments in the rear of the lamellipodium network due to the absence of some stabilizing chemicals during extraction. We selected three crops I, II and III, to study from this image, where the actin filaments are sparse, half-sparse and half-dense, and dense, respectively. We then thresholded each cropped image using the Otsu thresholding method to segment filaments, and restricted the original image pixel intensities to these filaments to finally obtain a discrete probability distribution. Points were sampled from this distribution and their spatial coordinates were perturbed with a Gaussian noise of standard deviation $10\%$ of a pixel side.

We show the results of our computations in Figures \ref{fig:actin_stained_cycles}, \ref{fig:actin_overlayed_cycles}, \ref{fig:actin_barcodes}, \ref{fig:actin_prevalence_barcodes},
\ref{fig:actin_scores}, and
\ref{fig:actin_persist_vs_preval}.
We found that larger voids in the structure of the actin mesh led to larger cycles of higher prevalence, such as in crop I, whereas small voids in denser parts led to smaller cycles of lower prevalence, such as in crop III. As seen in Figures \ref{fig:actin_stained_cycles} and \ref{fig:actin_overlayed_cycles}, we correctly identified large cycle representatives in crop I, small ones in crop III, and a mix of small and large ones in crop II, at the transitional region between rear and front of the lamellipodium. The usual persistence barcodes (Figure \ref{fig:actin_barcodes}) show numerous short-lived bars in crop III, but some longer-lived features for crops I and II. This barcode information can be enriched by visualizing the prevalence as the thickness and color of a bar (Figure \ref{fig:actin_prevalence_barcodes}). It is interesting to note that the highest prevalence scores throughout selected data were found in crop II, corresponding to two highly prevalent cycles (see Figure \ref{fig:actin_scores}). Their scores are higher than those from crop I, due to contrast variations of larger amplitude along filaments of crop II. Indeed, prevalence can be interpreted as a certainty measure of topological features.
Finally, a scatter plot of prevalence versus persistence scores (Figure \ref{fig:actin_persist_vs_preval}) confirmed that prevalence is not monotonous with respect to persistence and longer intervals do not necessarily correspond to more prevalent features.

\subsection{A Note on Computational Runtime}

As mentioned previously in section \ref{subsec:implementation}, the advantage of the framework proposed by \cite{reani_cycle_2021} is that the procedure of matching is easily parallelizable. With access to standard institutional high performance computing (HPC) resources requiring simply CPU processing, use of a single node, one CPU per task, and max 30 GB of memory per CPU, the problem of identifying prevalent cycles or matching intervals in spaces with a range of 100--1000 points generally reduces to a matter of minutes, ranging from seconds to a few hours for the applications showcased here.

We present in Table \ref{tab:runtime_tracking} the runtimes corresponding to the real datasets from Section \ref{sec:track_slices} and Section \ref{sec:track_video}, where we track topological features over a set of frames. We include the number of points in the samples \(N\) and the number of subsamples \(K\). In Table \ref{tab:runtime_real}, we exhibit the runtimes associated to the real datasets shown in Section \ref{sec:appli_prevalent}, where we compute prevalent features. We include the number of points on the reference space \(N_\mathrm{ref}\), the number of points in the resamples \(N\), and the number of resamples \(K\). Table \ref{tab:runtime_synth} collects runtimes for computing prevalent features on synthetic datasets that consist of point clouds uniformly sampled in the unit square of the plane. Note that we took $N_\mathrm{ref} = N$ in the synthetic examples. 

In Table \ref{tab:runtime_tracking}, one runtime corresponds to the computation of 
\begin{enumerate}
\item the barcodes of the two samplings on consecutive frames;
\item the two image-barcodes of those samplings in their union; and
\item the matching itself, which compares the barcodes. 
\end{enumerate}

In Table \ref{tab:runtime_real} and Table \ref{tab:runtime_synth}, step (i) only covers the computation of the barcode on the resampling \(X^{(k)}\) corresponding to that job, and in step (ii), we compute the image-persistence of the reference space \(X_\mathrm{ref}\) and the resampling \(X^{(k)}\) in their union. The runtime needed to compute the barcode of the reference space $X_\mathrm{ref}$, which is just a few seconds and is computed once before the parallel jobs, is not included in Tables \ref{tab:runtime_real} and \ref{tab:runtime_synth}. 

The computational bottleneck here is not the number of matchings $K$ but the number of points $N$ and $N_\mathrm{ref}$ sampled in the respective spaces for each matching instead. The number of matchings could be increased arbitrarily (up to the capacity of the HPC)
without affecting the computational runtime while maintaining it on the order of minutes or a few hours. Increasing it here allows for more precise estimations of, for instance, the average runtime needed to find all possible matchings between two spaces.
However, median runtime increases following a power law $T \sim \mathrm{cst} \cdot N^p$ where $\mathrm{cst}$ is a constant and $p \simeq 3.266$ (see Figure \ref{fig:runtime_regression}), so that increasing the number of sampled points $N$ by a factor of $10$ would result in increasing runtime by a factor of $10^p \simeq 1845$.

\begin{table}[!ht]
    \centering
    \begin{tabular}{|c|c|c|c|c|c|c|c|}
    \hline
        \multirow{2}{*}{Dataset} & \multicolumn{4}{|c|}{Runtime (minutes)} & \multicolumn{2}{|c|}{Parameters} \\ \cline{2-7}
        & min & max & average & median & $N$  & $K$ \\ \hline
        Lateral Line  & 31 & 281 & 124 & 95 & 1000 & 14 \\ \hline
        Heart Valves & 3 & 21 & 11 & 8 & 500 & 9\\ \hline
        Time-lapse Embryo & 5 & 31 & 15 & 15 & 500 & 9\\ \hline
    \end{tabular}
    \caption{Computational Runtimes of Tracking Experiments on Real Data.}
    \label{tab:runtime_tracking}
\end{table}

\begin{table}[!ht]
    \centering
    \begin{tabular}{|c|c|c|c|c|c|c|c|}
    \hline
        \multirow{2}{*}{Dataset} & \multicolumn{4}{|c|}{Runtime (minutes)} & \multicolumn{3}{|c|}{Parameters} \\ \cline{2-8}
        & min & max & average & median & $N_\mathrm{ref}$ & $N$ & $K$ \\ \hline
        Actin Crop I & 83 & 360 & 184 & 178 & 1200 & 500 & 30 \\
        Actin Crop II & 58 & 377 & 146 & 118 & 1200 & 500 & 30 \\ 
        Actin Crop III & 65 & 310 & 164 & 172 & 1200 & 500 & 30 \\ \hline
        Cosmic Web & 8 & 29 & 15 & 14 & 1000 & 300 & 20 \\ \hline
    \end{tabular}
    \caption{Computational Runtimes of Prevalence Experiments on Real Data.}
    \label{tab:runtime_real}
\end{table}

\begin{table}[!ht]
    \centering
    \begin{tabular}{|c|c|c|c|c|c|c|c|}
    \hline
        \multirow{2}{*}{Dataset} & \multicolumn{4}{|c|}{Runtime (seconds or minutes)} & \multicolumn{3}{|c|}{Parameters} \\ \cline{2-8}
        & min & max & average & median & $N_\mathrm{ref}$ & $N$ & $K$ \\ \hline
        Synth. 1 & 1.81 s & 5.61 s & 3.43 s & 3.34 s & 100 & 100 & 30 \\
        Synth. 2 & 8.99 s & 48.32 s & 23.47 s & 20.17 s & 200 & 200 & 30 \\ 
        Synth. 3 & 25.68 s & 2 mn 44 s & 1 mn 25 s & 1 mn 18 s & 300 & 300 & 30 \\
        Synth. 4 & 3 mn 13 s & 15 mn 31 s & 8 mn 12 s & 8 mn 5 s & 500 & 500 & 30 \\
        Synth. 5 & 13 mn 19 s & 88 mn 54 s & 46 mn 18 s & 39 mn 48 s & 800 & 800 & 30 \\
        Synth. 6 & 38 mn 41 s & 189 mn 33 s & 95 mn 16 s & 92 mn 10 s & 1000 & 1000 & 30 \\ \hline
    \end{tabular}
    \caption{Computational Runtimes of Prevalence Experiments on Synthetic Data. Median runtime increases as a power law $T \sim \mathrm{cst} \cdot N^{3.266}$ with respect to $N$ (see Figure \ref{fig:runtime_regression}).}
    \label{tab:runtime_synth}
\end{table}

\subsection*{Software and Data Availability}

The code used to perform all experiments here is freely and publicly available at our project GitHub repository \url{https://github.com/inesgare/interval-matching}.  It is fully adaptable for individual user customization.  

Where possible, the data we have used in this section are also provided on the same GitHub repository so that all experiments and examples in our paper are fully reproducible.  Note that some of the data we have used here required an institutional materials transfer agreement, so these data were not made available on our repository.

\section{Discussion}
\label{sec:end}

In this paper, we studied the problem of identifying topologically significant features in noisy data where the usual measure of persistence in the sense of the length of an interval in a persistence barcode is unsatisfactory.  We also studied the problem of comparing barcodes over different filtrations and identifying correspondences between persistence intervals.  To date, the various existing proposals to these problems have faced significant computational limitations.  The main contribution of our work is an extension of existing notions of topological significance and cycle matching to provide the most general and flexible definitions, which we then implement using the dual perspective of cohomology to achieve the fastest available identification of prevalent cycles as well as cycle matching. Our implementation now makes these approaches practical and applicable to real-life, large-scale, complex datasets with execution times ranging from a matter of minutes to a few hours for 100--1000 sampled points, 
using only standard institutional HPC facilities.
Our work inspires several directions for future research, which we now discuss.  

First, one natural question is to understand the behavior of the prevalence-augmented barcodes as the number of resampling spaces $K$ and the number of sampling points $N$ for each space increase to infinity. It will be important to understand how at the limit augmented barcodes are related to the original distribution, quantify the rate and type of convergence, and study how the choice of filtration (Rips, \v{C}ech, coupled-Alpha \citep{reani_coupled_2021}) and of affinity (A, B, C, D) affects the convergence, if at all. This would then allow us to design probabilistically-founded statistical tests, for example, to determine whether a collection of point clouds has been sampled from a specific distribution with known barcode, or not, based on some confidence intervals.
Developing such a framework would provide a practical approach to choosing a threshold to identify ``true'' cycles in the original data as bars whose prevalence are higher than $1 - \epsilon$.  This could also have practical implications on certain applications, for example, by directly extracting ``true'' topological signal in complex settings such as imaging data, and perhaps even bypassing the need for computationally expensive procedures, such as image segmentation.

Another theoretical question of interest is the study of the new metric $d_{\mathrm{IM}_p}$ on persistence modules introduced by \cite{reani_cycle_2021} in section 6.2. This metric involves comparing matched intervals between two modules directly, and not comparing birth--death times between persistence diagrams which discards important information about spatial correspondence, as explained previously. It is reasonable to expect that as $N$ increases, the distance between two persistence modules converges to zero if the point clouds are drawn from the same distribution. Likewise, this could be an alternative approach towards the design of a (non-parametric)
statistical test to determine whether two points clouds were sampled from a same distribution.

\section*{Acknowledgments}

We wish to thank Omer Bobrowski, Dominique Bonnet, Vin de Silva, Bernhard Kainz, Yohai Reani, Wojciech Reise, Marc Glisse, and Iris Yoon for helpful conversations.  We also wish to thank two anonymous referees for their insightful comments and help with improving the paper.

I.G.R.~is funded by a London School of Geometry and Number Theory--Imperial College London PhD studentship, which is supported by the Engineering and Physical Sciences Research Council [EP/S021590/1], EPSRC Centre for Doctoral Training in Geometry and Number Theory (The London School of Geometry and Number Theory), University College London.  A.S.~is funded by a joint Imperial College London--Francis Crick Institute PhD studentship.

We wish to acknowledge The Francis Crick Institute, London, and the Information and Communication Technologies resources at Imperial College London for their computing resources which were used to implement the experiments and data applications in this paper.

We would like to thank the projects involved in \cite{10.7554/eLife.55913}, \cite{scherz_developing_hearts_2008} and \cite{GOMEZ2022108258} for these datasets and their public accessibility.

We are grateful to the contributors of the Cell Image Library (CIL) \citep{ellisman_cell_2021}, located at \url{http://www.cellimagelibrary.org}, for making their data repository and resources available for public access.

Funding for SDSS-III has been provided by the Alfred P.~Sloan Foundation, the Participating Institutions, the National Science Foundation, and the U.S. Department of Energy Office of Science. The SDSS-III web site is \url{http://www.sdss3.org/}.

SDSS-III is managed by the Astrophysical Research Consortium for the Participating Institutions of the SDSS-III Collaboration including the University of Arizona, the Brazilian Participation Group, Brookhaven National Laboratory, Carnegie Mellon University, University of Florida, the French Participation Group, the German Participation Group, Harvard University, the Instituto de Astrofisica de Canarias, the Michigan State/Notre Dame/JINA Participation Group, Johns Hopkins University, Lawrence Berkeley National Laboratory, Max Planck Institute for Astrophysics, Max Planck Institute for Extraterrestrial Physics, New Mexico State University, New York University, Ohio State University, Pennsylvania State University, University of Portsmouth, Princeton University, the Spanish Participation Group, University of Tokyo, University of Utah, Vanderbilt University, University of Virginia, University of Washington, and Yale University.

\section*{Declarations}

The authors have no relevant financial or non-financial interests to disclose.
The authors have no competing interests to declare that are relevant to the content of this article.
All authors certify that they have no affiliations with or involvement in any organization or entity with any financial interest or non-financial interest in the subject matter or materials discussed in this manuscript.
The authors have no financial or proprietary interests in any material discussed in this article.

\clearpage

\section{Figures}
\begin{figure}[!ht]
    \centering
    \includegraphics[width=0.6\linewidth]{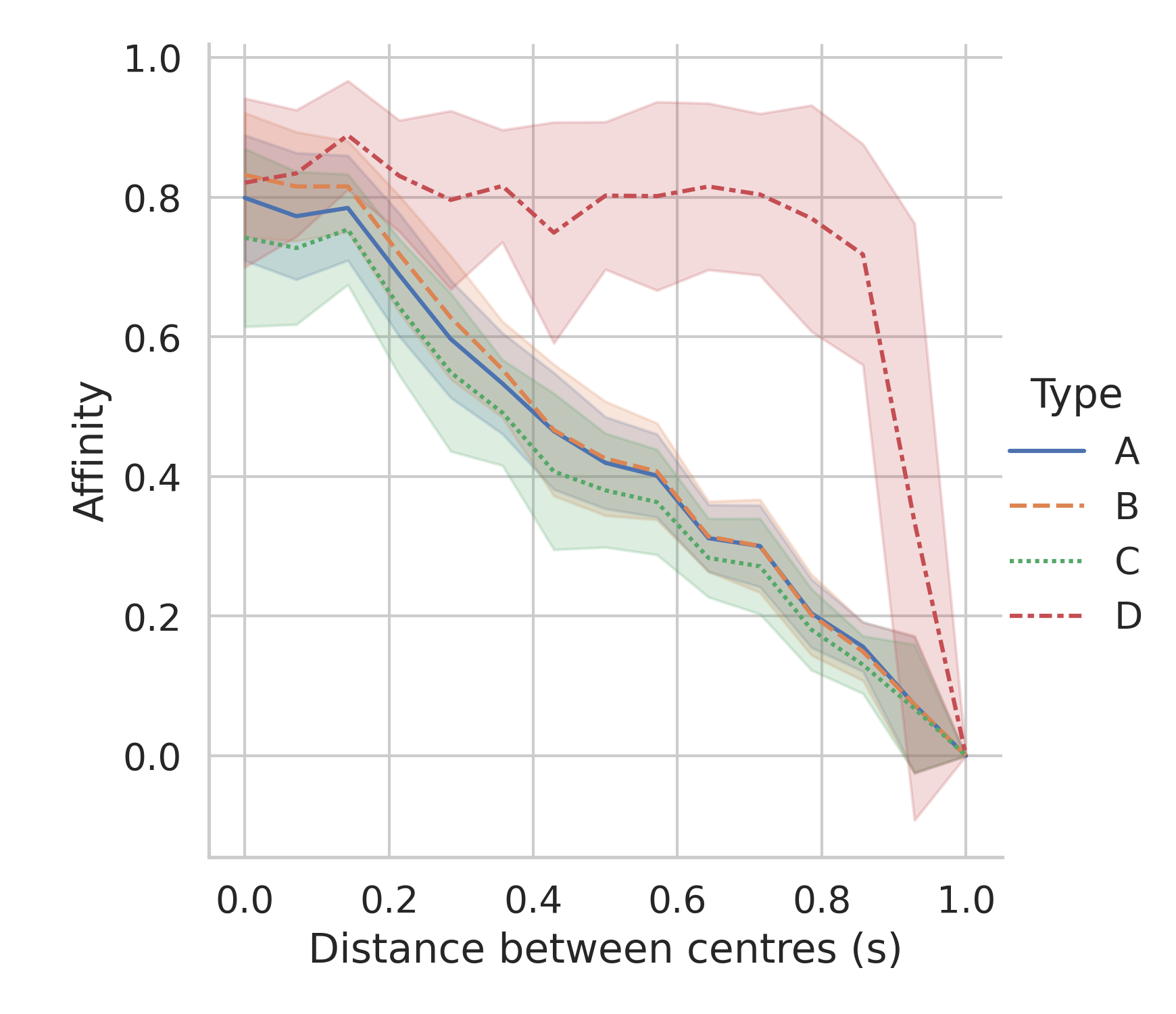}
    \caption{Mean and standard deviation of the affinities of the matches between two circles of radius 1 with centers shifted according to the horizontal axis. The circles were sampled with \(N = 100\) points without noise added. We considered 15 equidistant distances between 0 and 1 and took 15 samples at each step.}
    \label{fig:comparison_affinities}
\end{figure}

\begin{figure}[!ht]
\centering
\begin{subfigure}{0.45\linewidth}
    \includegraphics[width=\linewidth]{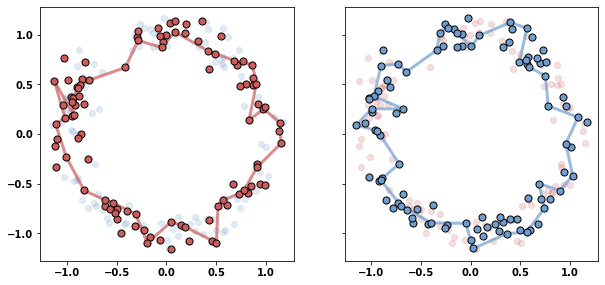}
    \caption{A match with affinity \(\rho_A = 0.9435\).}
    \label{subfig:first_match}
\end{subfigure}
\hfill
\begin{subfigure}{0.5\linewidth}
    \includegraphics[width=\linewidth]{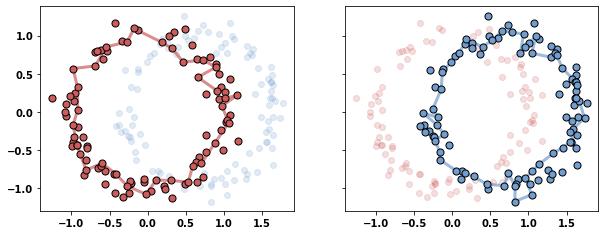}
    \caption{A match with affinity \(\rho_A = 0.2724\).}
    \label{subfig:second_match}
\end{subfigure}
\hfill
\begin{subfigure}{0.45\linewidth}
    \includegraphics[width=\linewidth]{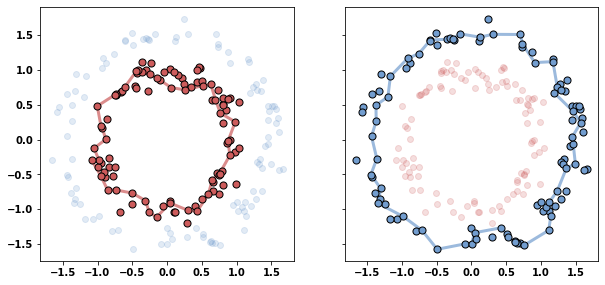}
    \caption{A match with affinity \(\rho_A = 0.2649\)}
    \label{subfig:third_match}
\end{subfigure}
\caption{Three matches between samples of 100 points of two circles with Gaussian noise of magnitude 0.1 added. In Figure \ref{subfig:first_match}; the circles have the same radii and centres; in Figure \ref{subfig:second_match} the circles have the same radii and centres 0.7 units of length apart; and in Figure \ref{subfig:third_match}, the circles have coinciding centres and radii 1 and 1.5.}
\label{fig:3_matches}
\end{figure}

\begin{figure}[!ht]
    \centering
    \includegraphics[width = \linewidth]{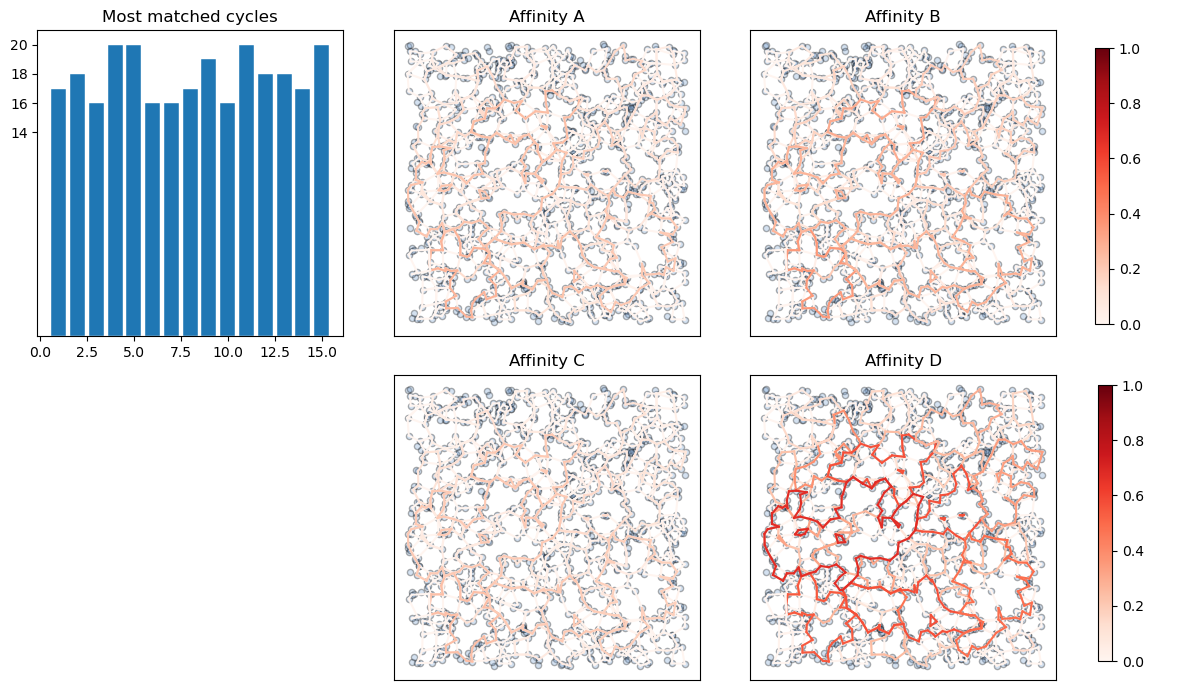}
    \caption{Most matched intervals in resamplings of \(N=1000\) points of the uniform distribution in the unit square. Top left: Frequency of reappearance of the 15 most frequently matched intervals. Remainder of the figure: Cycles representing the persistence intervals of \(X_{\mathrm{ref}}\) stained by their prevalence score using the affinity specified above of the image.}
    \label{fig:random_cycles}
\end{figure}

\begin{figure}[!ht]
    \centering
    \includegraphics[width = \linewidth]{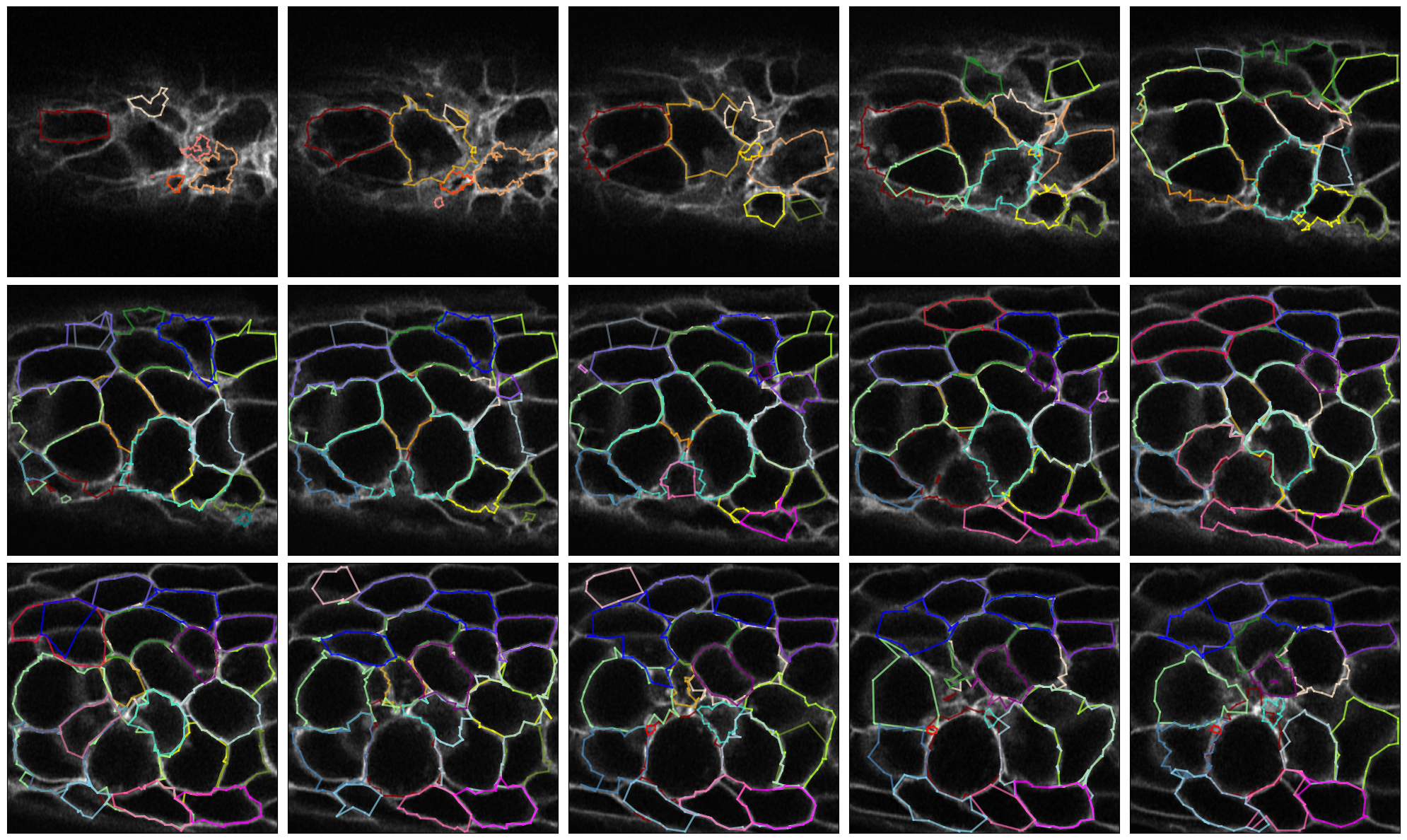}
    \caption{Cycle matching to track cell contours on images of slices of the posterior Lateral Line Primordium (pLLP) of the zebrafish. We applied an Otsu threshold and took samples of \(N = 1000\) points on the images. Cycles matched across consecutive slices can be grouped into tunnels, each stained with a different color. Data courtesy of \cite{10.7554/eLife.55913}.}
    \label{fig:2D_slices}
\end{figure}

\begin{figure}[!ht]
    \centering
    \includegraphics[trim = {0 0 0 0}, clip, width=\linewidth]{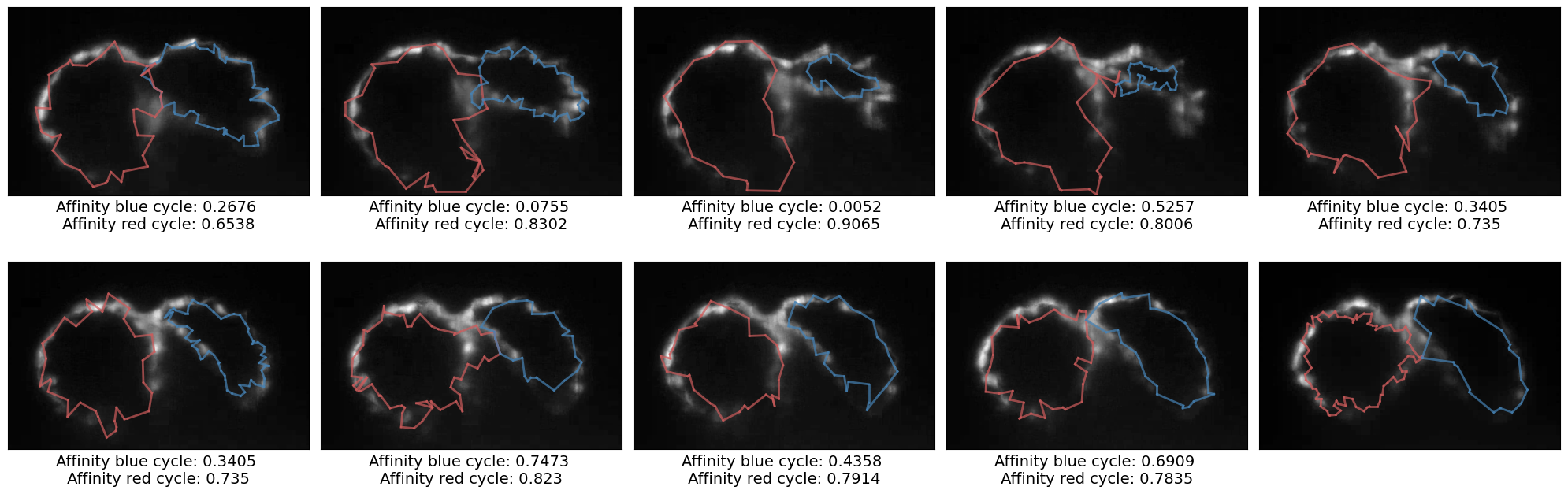}
    \caption{Cycle matching on 10 frames from a video courtesy of \cite{scherz_developing_hearts_2008}. We took samples of \(N = 500\) points and applied a threshold based on the mean of the gray-scale values before sampling. The intervals matched are stained in the same color. Below each image we display the affinity of the match between the interval in the image and the corresponding interval in the subsequent image.}
    \label{fig:heartbeat}
\end{figure}

\begin{figure}[!ht]
    \centering
    \includegraphics[width=\linewidth]{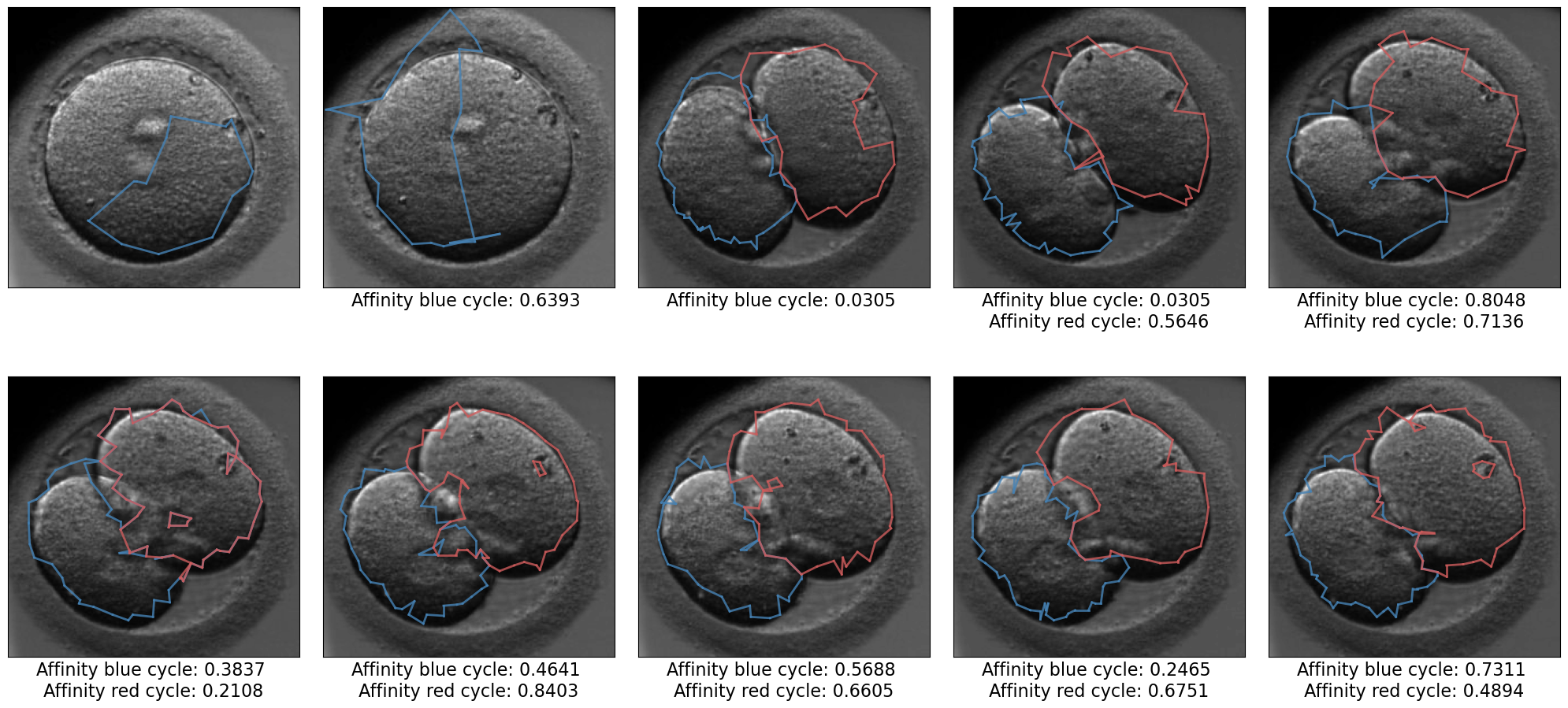}
    \caption{
    Cycle matching on the time-lapse embryo dataset \citep{GOMEZ2022108258} between samples of \(N = 500\) points in the images after applying a Sato operator and a threshold using the Otsu method. Cycles that get matched are stained in the same color. Below each image one can find the affinity of the match between the interval in the image and the corresponding interval in the previous image.}
    \label{fig:embryogenesis}
\end{figure}

\begin{figure}[!ht]
    \centering
    \includegraphics[trim = {0 2cm 0 2cm}, clip, width=.6\linewidth]{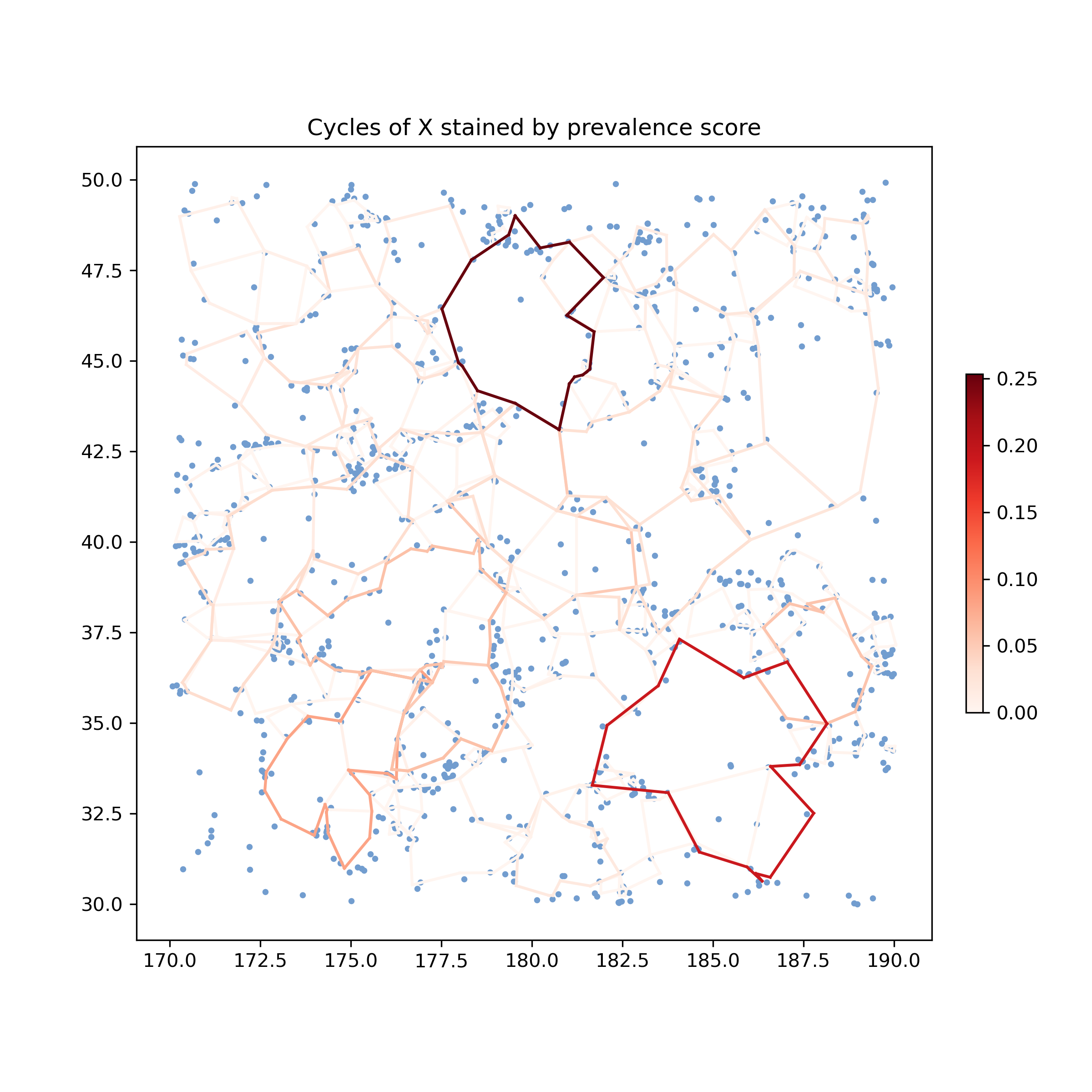}
    \caption{Prevalent cycles in the cosmic web. Cycle representatives are stained by prevalence score and galaxies (from the original BOSS CMASS data) are shown as blue dots. We formed the reference space by sampling the galaxies with $N_\mathrm{ref} = 1000$ perturbed points and performed $K = 20$ comparisons to spaces formed by sampling with $N = 300$ perturbed points each.}
    \label{fig:cosmic}
\end{figure}

\begin{figure}
    \centering
    \includegraphics[trim = {0 1cm 0 2cm}, clip, width=.8\linewidth]{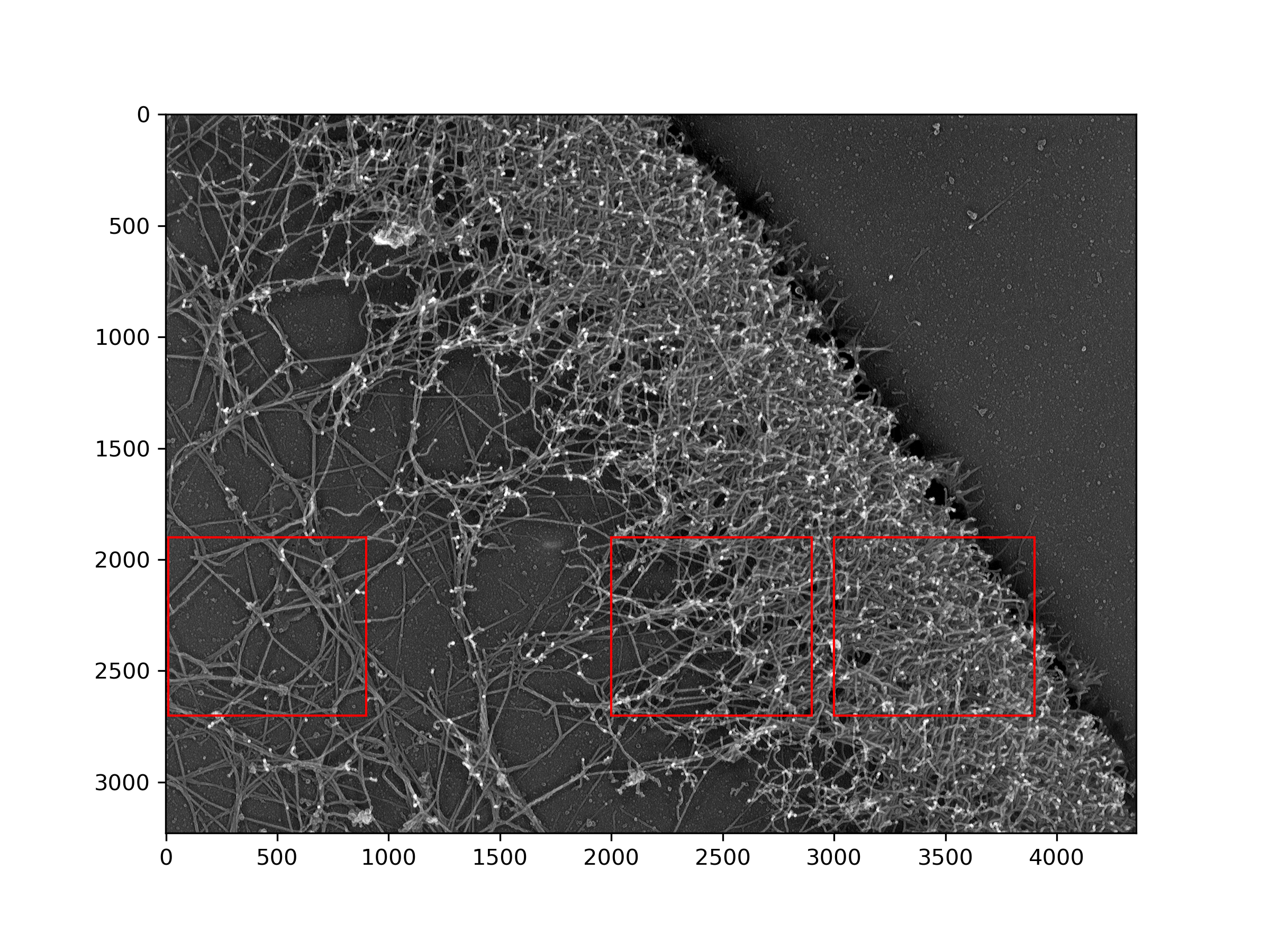}
    \caption{Electron micrograph of actin network in Xenopus keratocyte lamellipodium, whose rear part disassembled in the course of unprotected extraction and front part remained dense as in control cells. Selected crops I, II, III (from left to right) are shown as red rectangles. Original image CIL:24800 is from the Cell Image Library database \citep{ellisman_cell_2021}, available under CC BY-NC-SA 3.0 License, corresponding to Figure 6b of \cite{svitkina_arp23_1999}.}
    \label{fig:actin_data}
\end{figure}

\begin{figure}
     \centering
     \begin{subfigure}[b]{0.92\textwidth}
         \centering
         \includegraphics[trim = {0 0 0 0}, clip, width=1\textwidth]{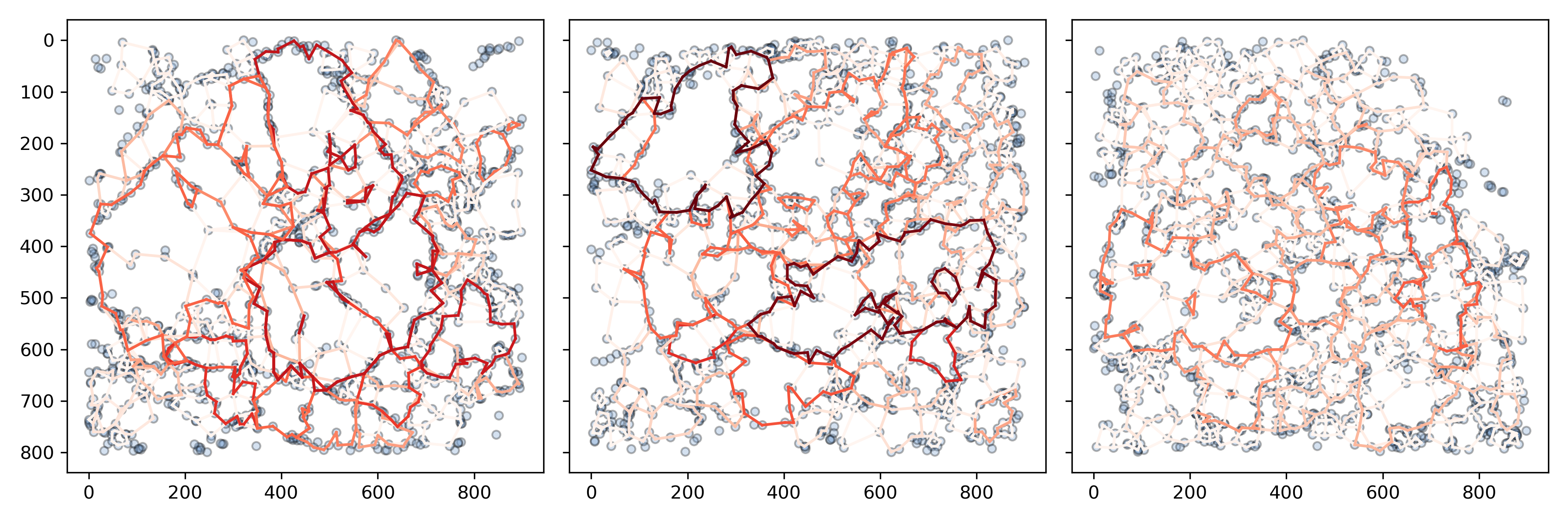}
     \end{subfigure}
     \hfill
     \begin{subfigure}[b]{0.06\textwidth}
         \centering
         \includegraphics[trim = {0 0 0 0}, clip,width=\textwidth]{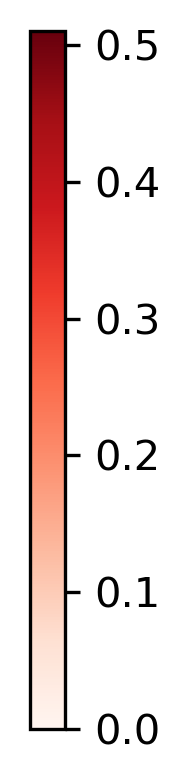}
         ~~~
     \end{subfigure}
        \caption{Prevalent cycles of reference space $X$ with $N_\mathrm{ref} = 1200$ points (shown as blue dots), based on $K = 30$ resampling spaces of $N = 500$ points. Cycle representatives are stained by prevalence score. From left to right: crops I, II, III. Colorbar describes prevalence scores.}
        \label{fig:actin_stained_cycles}
\end{figure}

\begin{figure}[!ht]
     \centering
     \begin{subfigure}[b]{0.92\textwidth}
         \centering
         \includegraphics[trim = {0 0 0 0}, clip, width=1\textwidth]{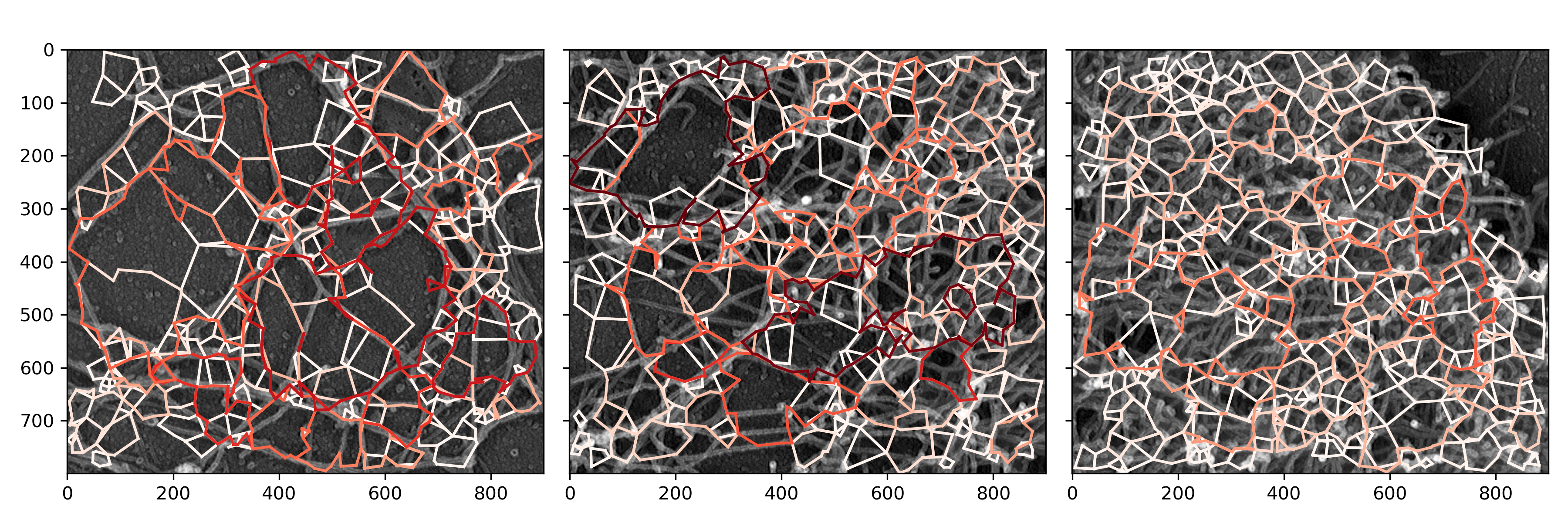}
     \end{subfigure}
     \hfill
     \begin{subfigure}[b]{0.06\textwidth}
         \centering
         \includegraphics[trim = {0 0 0 0}, clip,width=\textwidth]{actin/all_colorbar.png}
         ~~~
     \end{subfigure}
        \caption{Prevalent cycles overlayed on original image (see Figures \ref{fig:actin_data} and \ref{fig:actin_stained_cycles}). From left to right: crops I, II, III. Colorbar describes prevalence scores.}
        \label{fig:actin_overlayed_cycles}
\end{figure}

\begin{figure}[!ht]
     \centering     
     \begin{subfigure}[b]{\textwidth}
         \centering
         \includegraphics[trim = {0 0 1.8cm 0}, clip, width=1\textwidth]{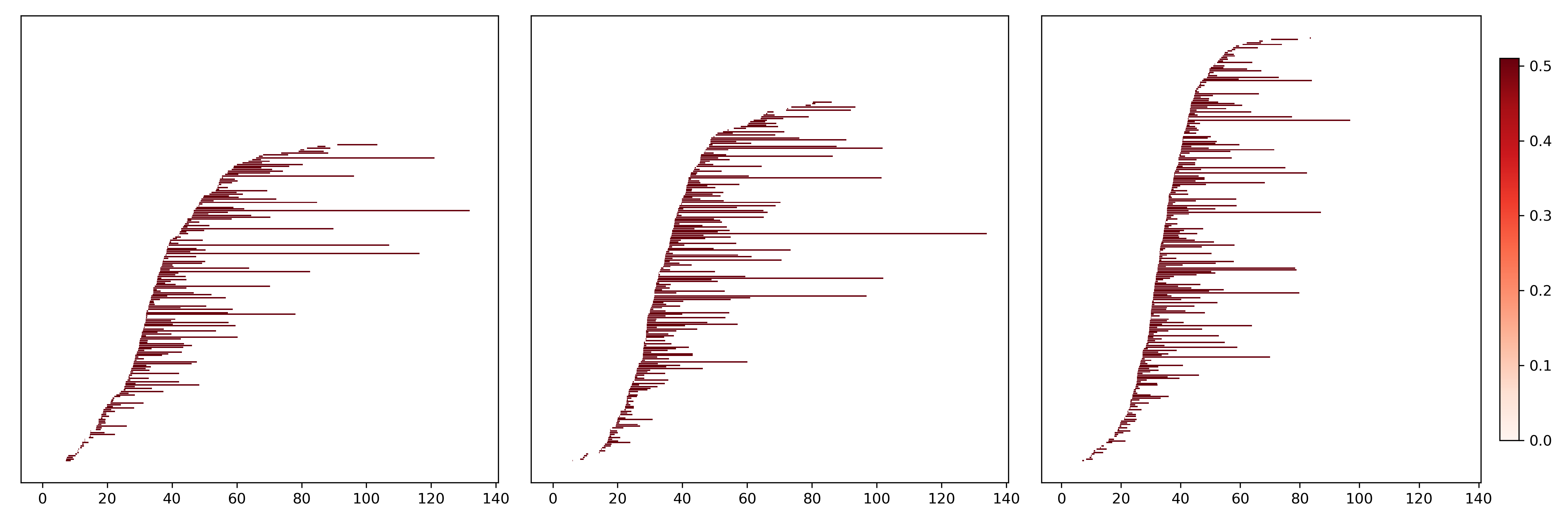}
     \end{subfigure}
        \caption{Persistence barcodes of reference space $X$. Values on the horizontal axis correspond to birth time and length of a bar to persistence. From left to right: crops I, II, III.}
        \label{fig:actin_barcodes}
\end{figure}

\begin{figure}[!ht]
     \centering
     \begin{subfigure}[b]{0.92\textwidth}
         \centering
         \includegraphics[trim = {0 0 1.8cm 0}, clip, width=1\textwidth]{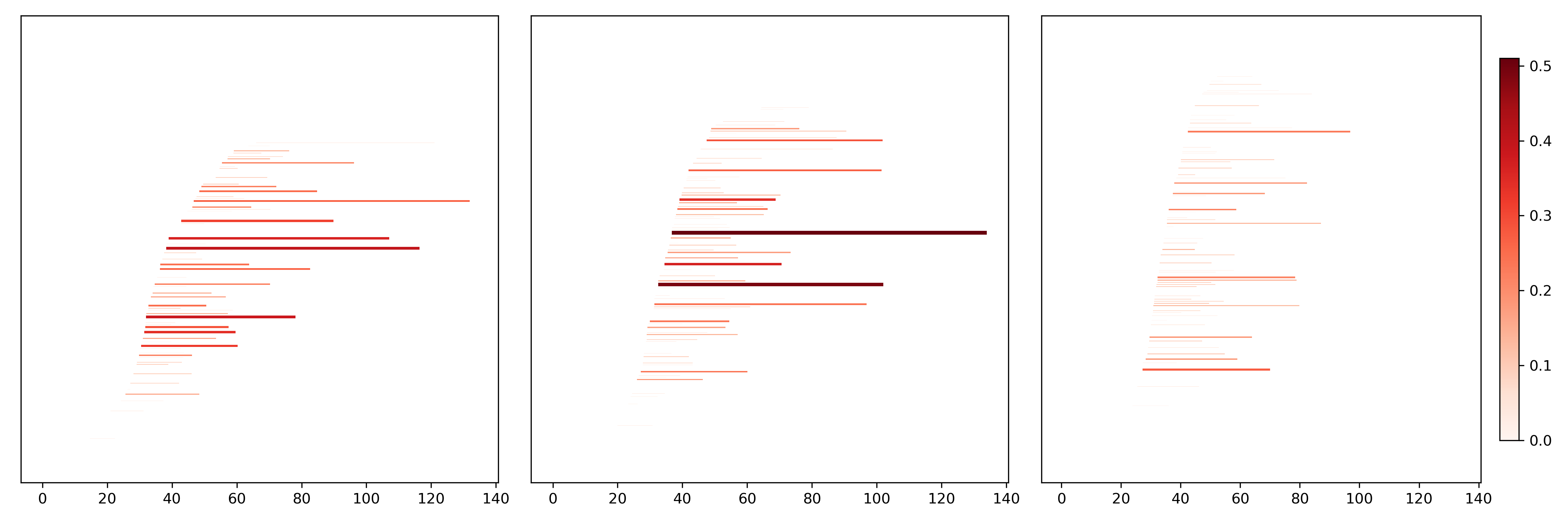}
     \end{subfigure}
     \hfill
     \begin{subfigure}[b]{0.06\textwidth}
         \centering
         \includegraphics[trim = {0 0 0 0}, clip,width=\textwidth]{actin/all_colorbar.png}
         ~~~
     \end{subfigure}
        \caption{Prevalence-augmented persistence barcodes of reference space $X$. Thickness and color of a bar correspond to prevalence score.
        Values on the horizontal axis correspond to birth time and length of the bars to persistence. From left to right: crops I, II, III. Colorbar describes prevalence scores.}
        \label{fig:actin_prevalence_barcodes}
\end{figure}

\begin{figure}[!ht]
    \centering
    \includegraphics[trim = {0 0 0 0}, clip, width=\linewidth]{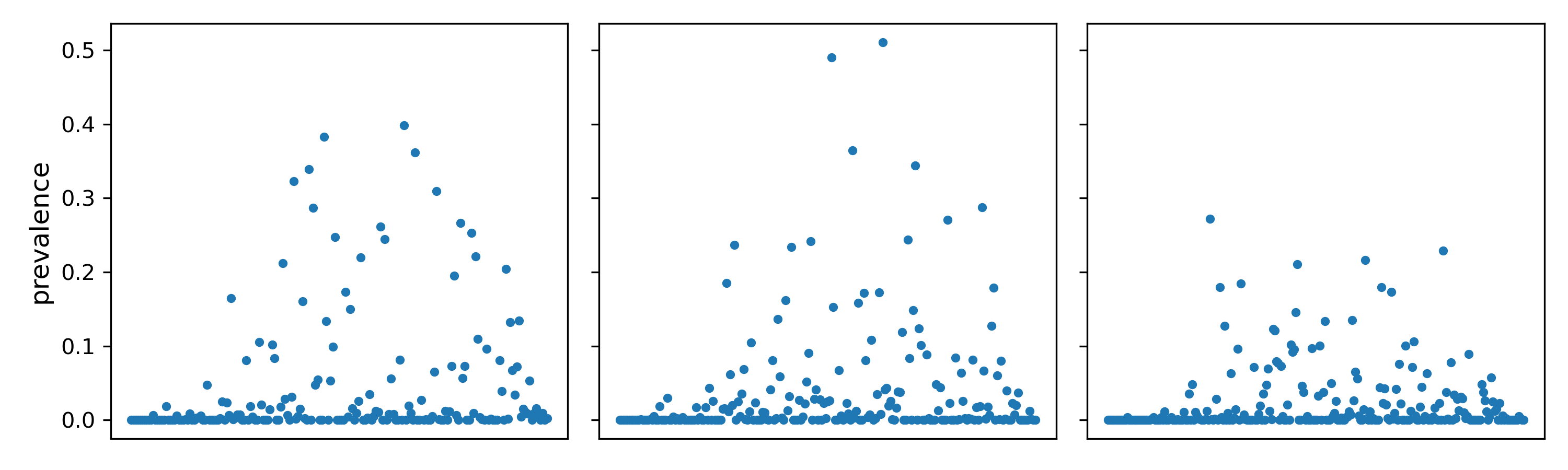}
    \caption{Prevalence scores, sorted by birth time of the persistence interval. One dot stands for one interval. From left to right: crops I, II, III.}
    \label{fig:actin_scores}
\end{figure}

\begin{figure}[!ht]
    \centering
    \includegraphics[trim = {0 0 0 0}, clip, width=\linewidth]{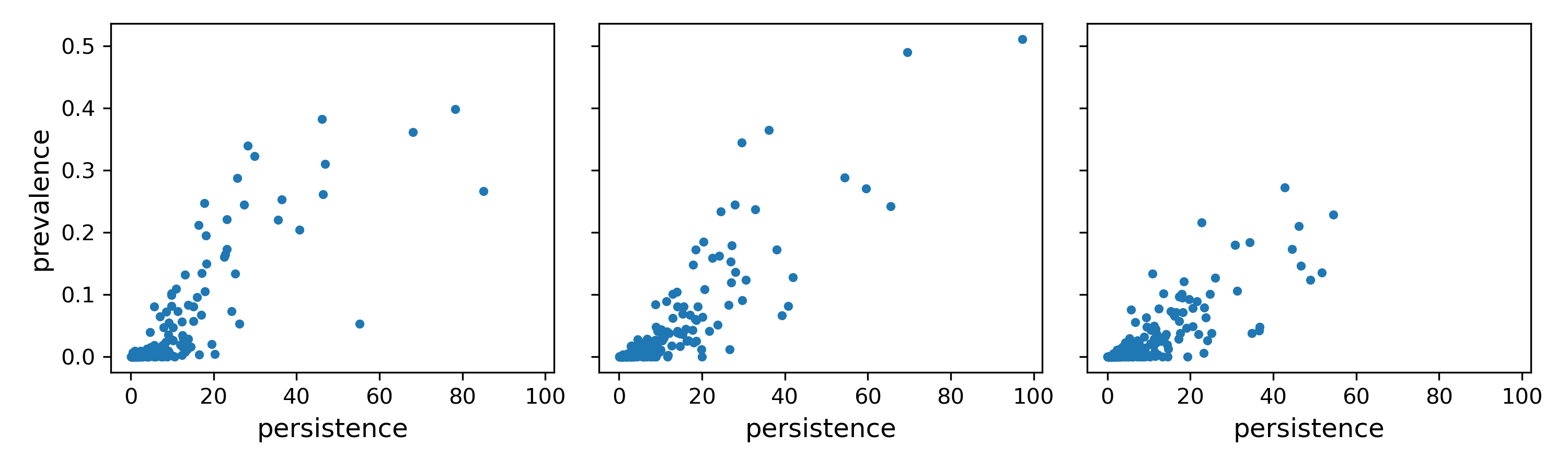}
    \caption{Persistence vs prevalence scores in the augmented barcode of the reference space $X$. One dot stands for one persistence interval. From left to right: crops I, II, III.}
    \label{fig:actin_persist_vs_preval}
\end{figure}

\begin{figure}[!ht]
    \centering
    \includegraphics[clip, width=\linewidth]{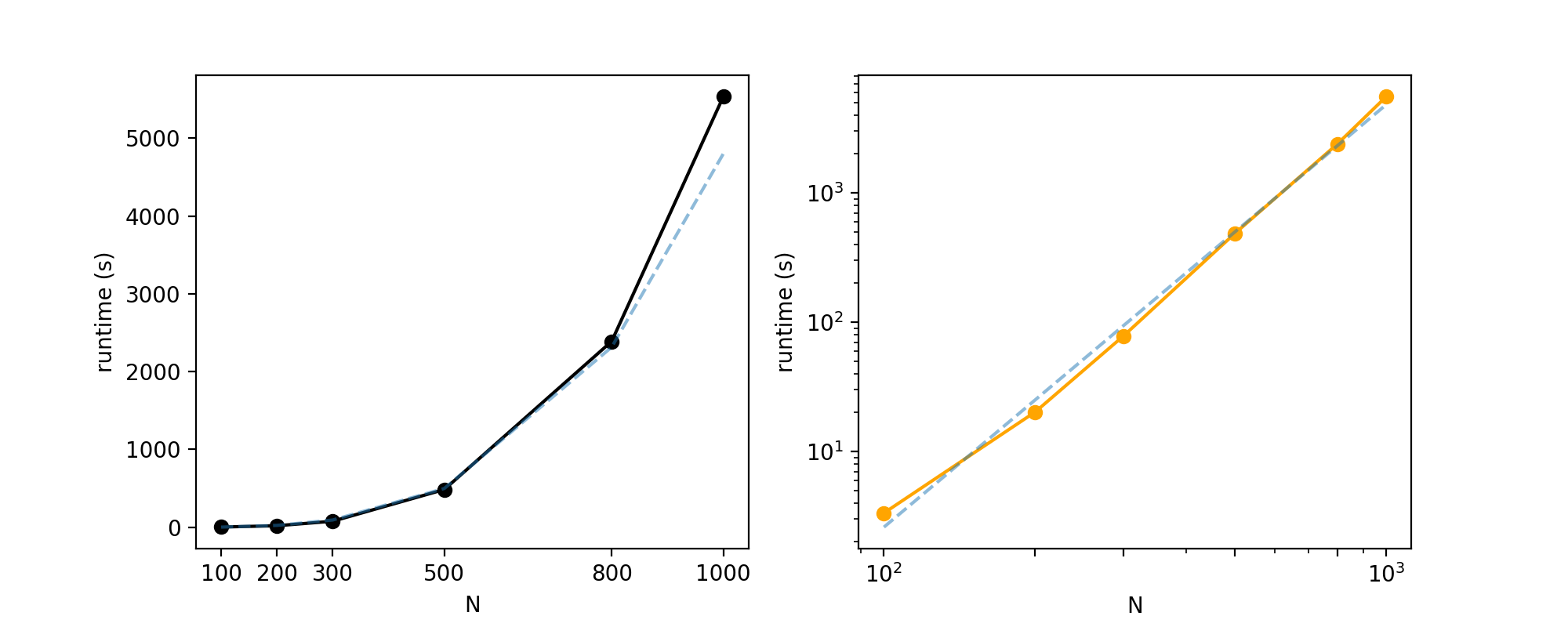}
    \caption{Median runtime v.s. number of sampling points $N$ for the synthetic examples of Table \ref{tab:runtime_synth}. Median runtime increases as a power law $T \sim \mathrm{cst} \, N^{3.266}$ with respect to $N$ (a linear regression on the logarithmic scale gave a score of $0.996$, coefficient $3.266$ and intercept $-14.087$ with respect to the data points of Table \ref{tab:runtime_synth}). Left: linear-linear scale. Right: log-log scale. The dashed line corresponds to the result of the linear regression on the log-log scale.}
    \label{fig:runtime_regression}
\end{figure}

\clearpage

\newpage

\bibliographystyle{authordate3}
\bibliography{coho_cycle_matching_ref}

\end{document}